\documentclass[12pt, letterpaper]{article}
\RequirePackage{latex_style}

\begin{document}


\title{\bf\color{titleblue}{A Singular Integral Measure for \\$C^{1,1}$ and $C^1$ Boundaries}}
\author[]{Laramie Paxton\thanks{realtimemath@gmail.com} } 
\author[]{Kevin R. Vixie\thanks{vixie@speakeasy.net}}
\affil[]{Department of Mathematics and Statistics\\  Washington State University}
\date{\today}
\renewcommand\Authands{\\ \ and }

\maketitle
\abstract{The art of analysis involves the subtle combination of approximation, inequalities, and geometric intuition as well as being able to work at different scales. With this subtlety in mind, we present this paper in a manner designed for wide accessibility for both advanced undergraduate students and graduate students. The main results include a singular integral for measuring the level sets  of a $C^{1,1}$ function mapping from $\R^n$ to $\R$, that is, one whose derivative is Lipschitz continuous. We extend this to measure embedded submanifolds in $\R^2$ that are merely $C^1$ using the distance function and provide an example showing that the measure does not hold for general rectifiable boundaries.}
\tableofcontents

\section{Introduction}
Often times it happens that creative exploration leads to uncharted territory in mathematics. That is indeed the case in the development of the singular integral boundary measure for hypersurfaces in $\R^n$, presented below.
To highlight this process, we shall first explore the simple analysis problem that led to our discovery of  this integral and then present our singular integral that measures the $(n-1)$-dimensional Hausdorff measure $\Hd^{n-1}$ of level sets of $C^{1,1}$ functions, that is, those whose derivative is Lipschitz continuous.



Continuing the process of exploration and generalization, we present the same singular integral as a $C^1$-boundary measure using the distance function in $\R^2$. The arguments here are markedly different from the $C^{1,1}$ case and follow a ``by-construction,'' barehanded approach due to the fact that in this case, we may have $\epsilon$-neighborhoods of the boundary in which the normals intersect no matter how small $\epsilon$ is (i.e. we no longer have \textit{positive reach}; see Definition \ref{reach}), so we cannot use the Area Formula (see Theorem \ref{area1}).

\section{A Simple Analysis Problem}
Consider the following simple analysis problem:\par

\begin{prob}\label{simple}
For $f(x): \R\rightarrow\R$, suppose that $f(x)$ is
differentiable, $f(0) = 0$, and $\exists \lambda>0$ such that $\forall x\in \R$

     \begin{equation}\label{lam}
     \left|f'(x)\right| \leq \lambda |f(x)|.
     \end{equation}

\noindent Prove that $f(x) = 0$ everywhere.
\end{prob}

\begin{proof}[Solution 1: ODE's]
Take the families of curves in $\R^2$  $g_C=Ce^{2\lambda x}$ and $h_C=Ce^{-2\lambda x}$ as all the solutions to the differential equations $ g' = 2\lambda g$ and
$ h' = -2\lambda h$, respectively. Cleary $f(x)$ cannot simply be one of these curves since none of them are ever 0. Note that the 2 is included here so that any solution $f(x)$ that satisfies (\ref{lam}) that intersects one of these curves must cross it. Indeed, without loss of generality, suppose for $g_C$ there exists an $x_0$ such that for all $x>x_0$, $f(x)=C_0e^{2\lambda x}.$ Then we find that 
$$\frac{2\lambda C_0e^{2\lambda x}}{C_0e^{2\lambda x}}\le \lambda\Rightarrow 2\le1,$$
a contradiction. A similar argument holds for $h_C$.

Now, for $g_C$, $f(x)$ can only cross from above to below as $x$ increases. If not, then $f'(x)>C2\lambda e^{2\lambda x}=2\lambda f \Rightarrow \frac{f'}{f}>2\lambda,$ also a contradiction. Similarly, $f(x)$ can only cross $h_C$ from below to above as $x$ increases. 

Now, suppose $f(x^*)>0$ for some $x^*$ (for $f(x^*)<0$, then $-f(x)$ satisfies (\ref{lam}) and is positive at $x^*$). Pick the corresponding $C^{+*}=f(x^*)/e^{2\lambda x^*}$ for $g_{C^{+*}}$ and $C^{-*}=f(x^*)/e^{-2\lambda x^*}$ for $h_{C^{-*}}$. These curves serve as ``fences'' ($g_{C^{+*}}$ for $x<x*$ and $h_{C^{-*}}$ for  $x>x*$) that $f$ cannot cross so that  $f\neq0.$ Thus we have a contradiction which implies that $f(x)\equiv0$ everywhere (see Figures \ref{fences1}, \ref{fences22}, and \ref{fences3}).

\end{proof}

\begin{figure} [H]
\begin{center}
    \scalebox{0.6}{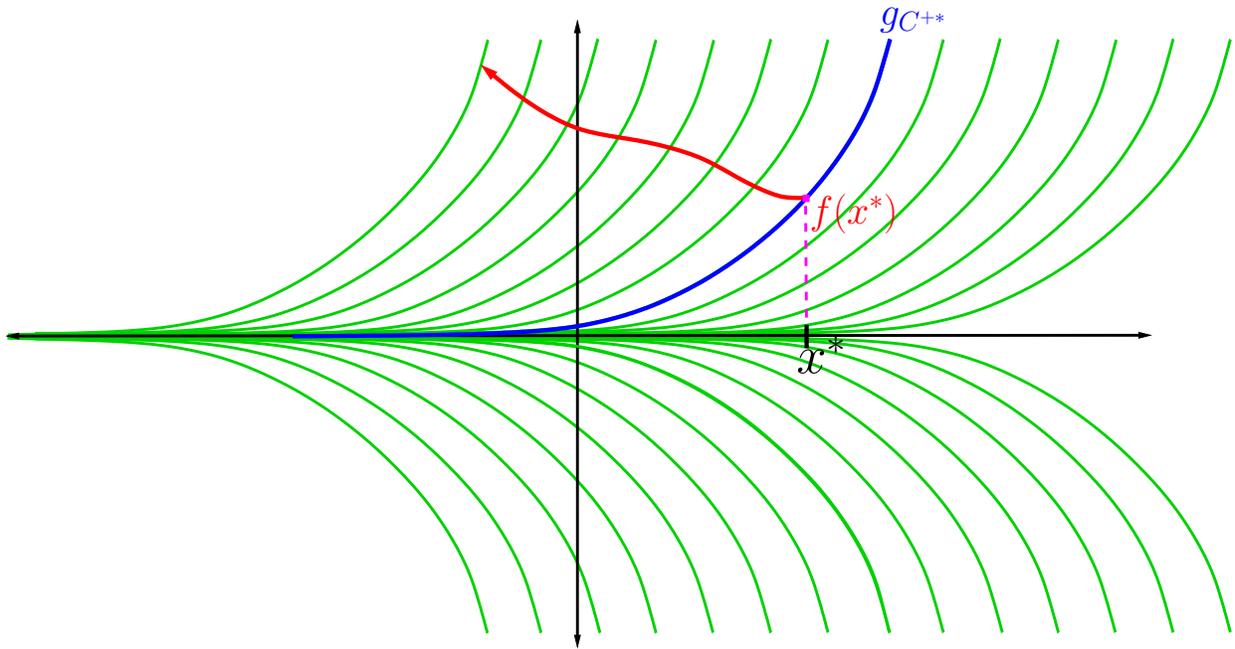}
      \caption[Constructing fences I.]{The case where $f(x^*)>0$ and $g_{C^{+*}}$ acts as a ``fence'' for $f(x)$ when  $x<x*$, bounding it away from 0.} 
      \label{fences1}
 \end{center}
\end{figure}

\begin{figure} [H]
\begin{center}
    \scalebox{0.6}{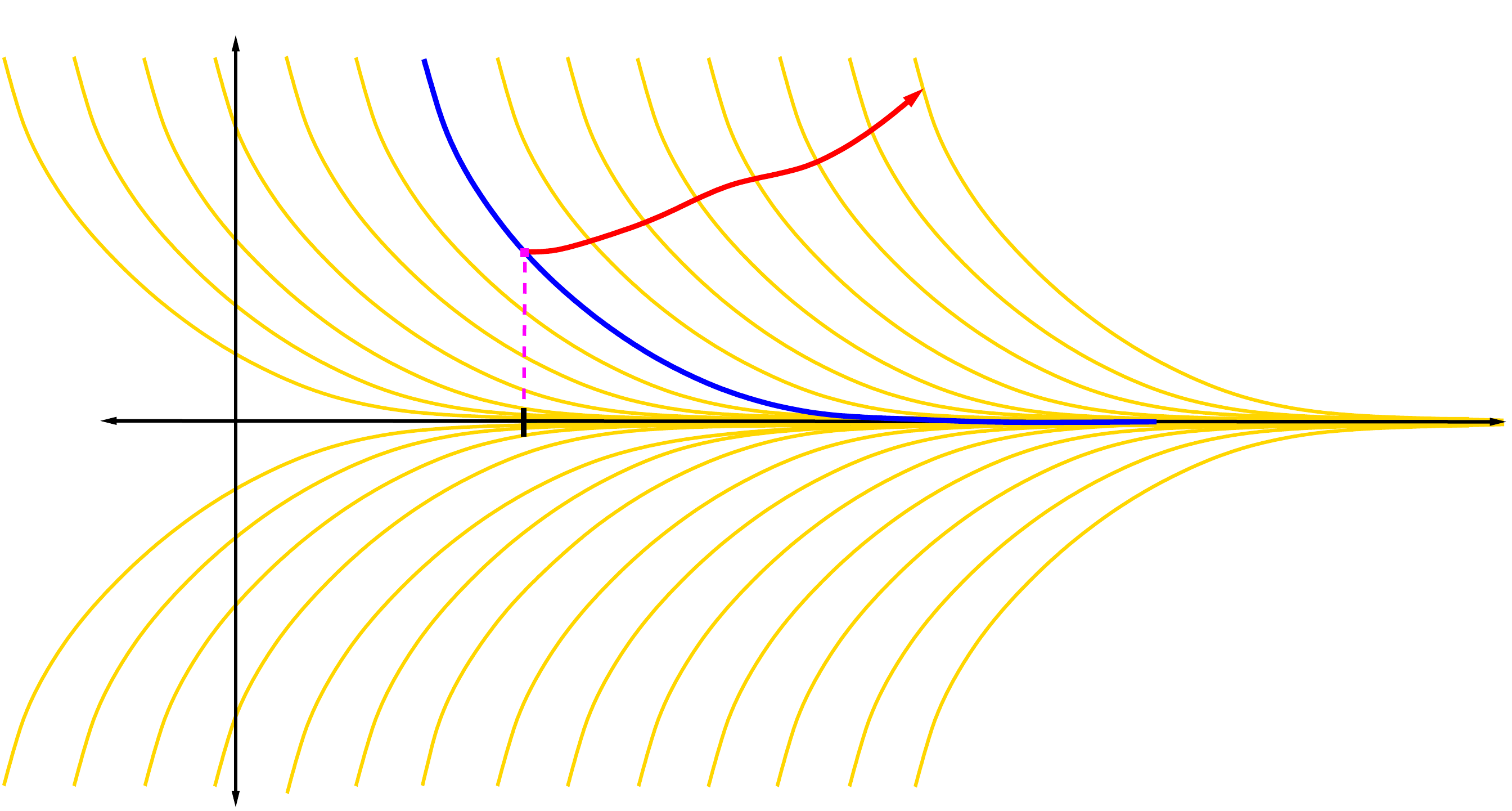}
      \caption[Constructing fences II.]{The case where $f(x^*)>0$ and $h_{C^{-*}}$ acts as a ``fence'' for $f(x)$ when  $x>x*$, bounding it away from 0.} 
      \label{fences22}
 \end{center}
\end{figure}

\begin{figure} [H]
\begin{center}
    \scalebox{0.68}{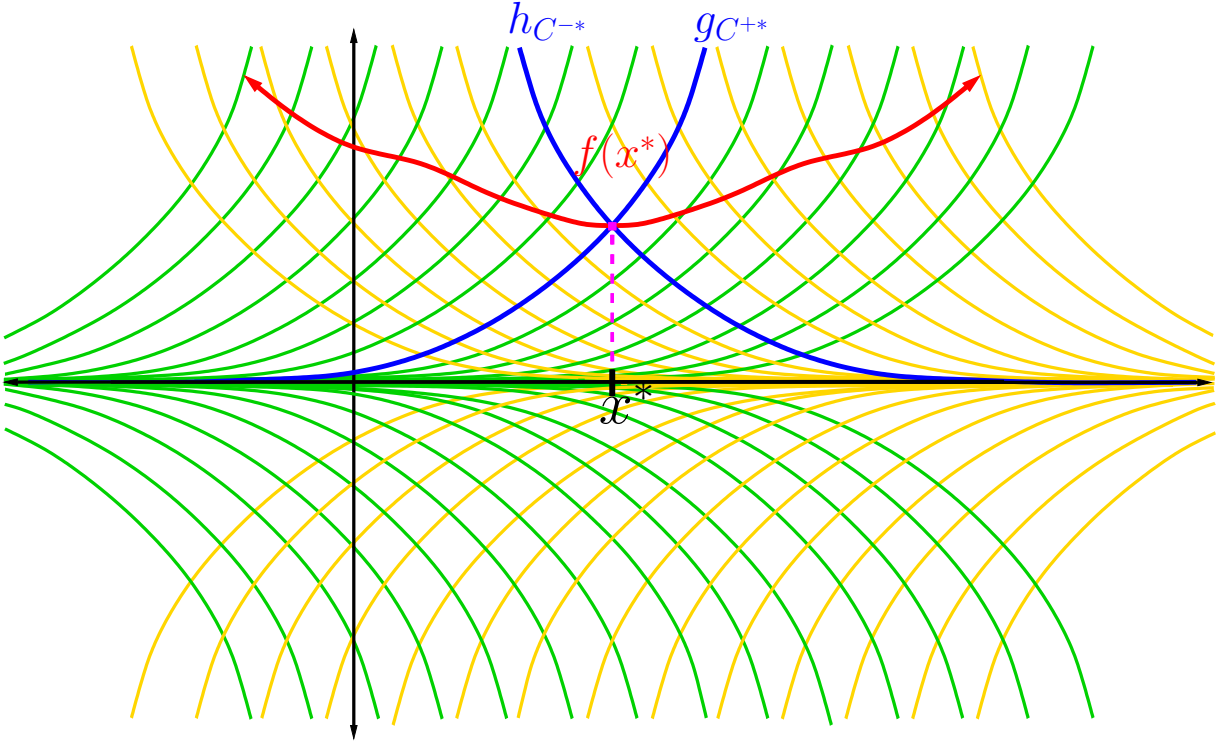}
      \caption[Constructing fences III.]{The two combined cases from above.} 
      \label{fences3}
 \end{center}
\end{figure}

\bigskip
\begin{proof}[Solution 2: Mean Value Theorem]
 Suppose $ f(x_0) = 0$. We prove $ f(x) = 0$
on $ I = [x_0 - \frac{1}{2\lambda}, x_0 +\frac{1}{2\lambda}]$. Cleary then we can pick the endpoints of this interval to extend our interval where $f(x)$ is 0. Continuing this process  thus this implies that $f(x)$ is 0 everywhere. 

 \noindent Let $ x\in I$. By the Mean Value Theorem,
$$ |f(x) - f(x_0)| \leq |f'(y_1)| |x - x_0| \leq |f'(y_1)|
\frac{1}{2\lambda},$$ 
for some $ y_1\in I$. By (\ref{lam}),
we have $$ |f(x)|=|f(x)-0| \leq |f'(y_1)|
\frac{1}{2\lambda} \le \frac{1}{2}f(y_1).$$ Similarly,  $ |f(y_1)| \leq \frac{1}{2}f(y_2)$ for
some $ y_2 \in I$. Thus $ |f(x)| \leq\frac{1}{2^{2}}f(y_2)$.  
Continuing this process, we obtain 
\begin{equation}\label{seq}
 |f(x)| \leq \frac{1}{2^{n}}f(y_n)               
\end{equation}
for some $ y_n \in I, n\in\mathbb{N}$. 

From the Extreme  Value Theorem, $ f(x)$ continuous $\Rightarrow  f(x) < M \text{ for all } x\in I$. Using this 
with (\ref{seq}), we get
$$ |f(x)| \leq \frac{M}{2^{n}} \text{ for all } n\in\mathbb{N}.$$
Sending $n\rightarrow \infty$ gives $f(x)\equiv0$.

 \end{proof}

  \begin{proof}[Solution 3: A Barehanded ODE Approach]    
We shall construct ``fences'' similar to the first solution above. If $f(x)>0 \ \forall x\in [x_0,x_1]\subset E\subset \R$, then $$ \left|f'(x)\right| \leq \lambda |f(x)| \  \Leftrightarrow \
 -\lambda \leq \frac{f'(x)}{f(x)} \leq \lambda. $$

  \noindent Integrating both sides of this over the interval $[x_0,x_1]$ gives 
 $$-\lambda (x_1 - x_0) \leq \ln\left(\frac{f(x_1)}{f(x_0)}\right) \leq
\lambda (x_1 - x_0).$$

 \noindent But this is equivalent to
\begin{equation}\label{e}
 e^{-\lambda (x_1 - x_0)} \leq \frac{f(x_1)}{f(x_0)} \leq e^{\lambda
(x_1 - x_0)}.
\end{equation}         

 Now, let $ f(x^*) > 0$. Define  $$ l = \inf \{ w | f(x) > 0 \text{
for all }w<x<x^* \}$$ and $$ u = \sup \{ w | f(x) > 0
\text{ for all } x^* < x < w\}.$$ We see that $ l \neq -\infty\Rightarrow f(l)  = 0$ and $ u \neq \infty\Rightarrow f(u)  = 0.$

Using  (\ref{e}), assume $ l \neq -\infty$. Without loss of generality, let $x_1 = x^*$ and  pick a sequence $ \{x_0\}
\downarrow l$. In the limit we find that $\infty\le C$, where $C$ is finite; hence a contradiction.

Or, we may assume  $ u \neq \infty.$ Let $x_0 = x^*$ and pick a sequence $ \{x_1\} \uparrow u$ to obtain a similar contradiction of $C\le0$, where $C$ is positive. 

In either case, we find that there is no $x^*$ such that $ f(x^*) > 0$. Recalling that if $f(x^*)<0$, then $-f(x)$ satisfies (\ref{lam}) and is positive at $x^*$, we find that $f(x)=0$ everywhere. 

\end{proof}

\noindent One interpretation of what we have shown so far is if
\begin{enumerate}
\item
$ f$ is differentiable,

\item$ f(x_0)=0$, and

\item for some $ \delta > 0$, we have that $ f(x) \neq 0$ when
$ x \neq x_0$ and $ x\in [x_0 - \delta, x_0 + \delta]$,
\end{enumerate}
then letting $A_{f}(x) \equiv
\left|\frac{f'(x)}{f(x)}\right|$,

$$ \limsup_{x\rightarrow x_0}A_{f}(x)=\infty.$$ Hence we see quite clearly that the ratio $\left|\frac{f'(x)}{f(x)}\right|$ detects roots. We shall rely on this property in our singular integral boundary measure introduced below.             





\section{Geometric Measure Theory Background}

In this section, we recall some important definitions and results for later use. In the next section, we present the first main result regarding the $\Hd^{n-1}$-measure of the 0-level sets of $C^{1,1}$ functions and compute the singular integral (\ref{integral}) for the simple example of a paraboloid whose 0-level set is the unit circle. 

\begin{defn}\label{haus}
\textbf{Hausdorff Measure.} With this outer (radon) measure, we can
  measure $k$-dimensional subsets of $\R^n$ $(k\leq n)$. While it is
  true that $\mathcal{L}^n=\Hd^n$ for $n\in
  \mathbb{N}$ (see Section 2.2 of \cite{evans-2015-measure}), Hausdorff measure $\Hd^k$ is also defined for
  $k\in[0,\infty)$ so that even sets as wild as fractals are
  \textit{measurable} in a meaningful way (see Figure \ref{hausd}). Note that $\Hd^0$ is the counting measure.\\

To compute the $k$-dimensional Hausdorff measure of $A\subset\R^n$:
\begin{enumerate}
\item Cover $A$ with a collection of sets $\mathcal{E}=\{E_i\}_{i=1}^\infty$, where diam$(E_i)\leq d \ \forall i.$ 
\item Compute the $k$-dimensional measure of that cover: $$\mathcal{V}_{\mathcal{E}}^k(A)=\sum_i\alpha(k)\left(\frac{\text{diam}(E_i)}{2}\right)^k,$$ where $\alpha(k)$ is the $k$-volume of the unit $k$-ball. 
\item Define $\Hd_d^k(A)=\inf_{\mathcal{E}}\mathcal{V}_{\mathcal{E}}^k(A)$, where the infimum is taken over all covers whose elements 
have maximal diameter $d$. 
\item Finally, we define $\Hd^k(A)=\lim_{d\downarrow0}\Hd_d^k(A).$
\end{enumerate}
\end{defn}

\begin{figure}[H]\label{hausd}
  \centering
  \scalebox{1}{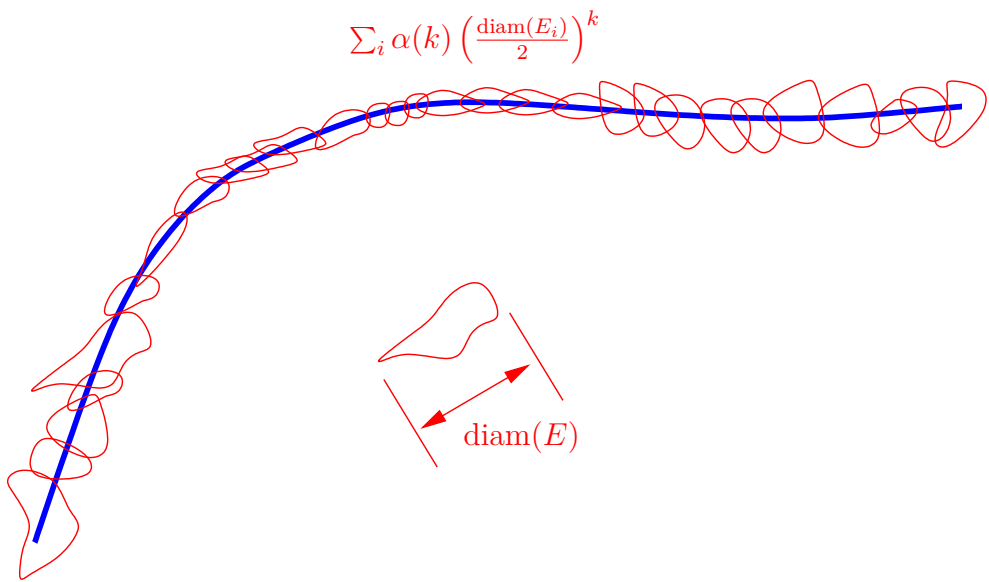}
  \caption[Hausdorff measure.]{The Hausdorff measure is derived from a cover of arbitrary sets.}
\end{figure}

Recall that  a function $f : \R^n\rightarrow \R^m$ is \textbf{Lipschitz continuous} if $\ \exists \alpha>0$ such that $|f(x)-f(y)|\leq \alpha |x-y|$ for all $x,y\in\R^n$, and we set $L\equiv \inf_\beta\{\alpha_\beta\}$ to be the Lipschitz constant of $f$. Also, recall that Lipschitz functions are differentiable almost everywhere by Rademacher's Theorem \cite{evans-2015-measure}. 

Thus, while a $C^1$ function is one which is continuously differentiable, a $C^{1,1}$ function is one whose derivative is Lipschitz continuous. By a ``$C^{1,1}$ set $E$,'' we shall mean that for all $x\in \partial E$, there is a neighborhood of $x$, $U_x\subset
\R^n$, such that after a suitable change of coordinates, there is a
$C^{1,1}$ function $f:\R^{n-1}\rightarrow\R$ such that $\partial E\cap
U_x$ is the graph of $f$.

\begin{defn}\label{reach}
 The  \textbf{reach} of $E$, $\reach(E)$, is defined
\[\reach(E)\equiv \sup \{r \; |   \text{ every ball of radius $r$ touching $E$ touches at a single point}\}.\]
 \end{defn}

\begin{rem}
If a set $E$ has \textit{positive reach}, then there exist $\epsilon$-neighborhoods $U_\epsilon$ of $E$ for all $\epsilon< \reach(E) $ such that for every $y\in U_\epsilon$ not in $E$, there exists a unique closest point $x\in E $, which is to say the normals $\vec{n}(x)$ to $E$ don't intersect in any $U_\epsilon$, where $\epsilon< \reach(E) $. 
\end{rem} 

 \begin{rem}
 If $\partial E$ is $C^{1,1}$, then $E$ has positive
reach (see Remark 4.20 in \cite{federer1959curvature}).
\end{rem} 

 \begin{thm}\label{imp}\cite{la2001characterization}\cite{luc1995taylor}
\textbf{($C^{k,1}$ Implicit Function Theorem)} Let $f:\R^n\times \R^m \rightarrow \R^m$ be a $C^{k,1}$ function, $k\geq 1$ with the property that $f(x_0,y_0)=0$ and suppose that $\nabla_yf(x_0,y_0)\neq 0.$ Then there exists an open neighborhood $U$ of $x_0$ and an open neigborhood $V$ of $y_0$ such that for any $x\in U$ there exists a unique $g(x)\in V$ with the property that $f(x,g(x))=0.$ The function $g:U\rightarrow V$ satisfies $g(x_0)=y_0$ and $g\in C^{k,1}$.
\end{thm}

The next theorem is a generalization of the \textit{change of variables formula} from calculus, the first part of which  says if $f$ behaves nicely enough, that is, $f$ is Lipschitz, then we can  calculate the $\Hd^n$-measure of $f(A)$ under suitable conditions.

\begin{thm}\label{area1}\cite{evans-2015-measure}
\textbf{(Area Formula I)} Let $f:\R^n \rightarrow \R^m$ be Lipschitz continuous, $n\leq m$, and $dx\equiv d\mathcal{L}^nx.$ Then for each $\mathcal{L}^n$-measurable subset $A\subset \R^n$,
\[
\int_A Jfdx=\int_{\R^m}\Hd^0(A\cap f^{-1}\{y\})d\Hd^ny.
\]
Moreover, if $\Omega$ is an $\mathcal{L}^n$-measurable subset of $ \R^n$, then for each $\mathcal{L}^n\myell \Omega$-integrable function $g(x)$,
\begin{equation*}
\int_{\Omega} g(x) Jf\,dx=\int_{f(\Omega)} \sum_{x\in f^{-1}\{y\}}g(x)d\Hd^ny.
\end{equation*}
\end{thm}

Here $Jf$ is the Jacobian of $f$, the $n$-volume expansion/contraction factor associated with the linear approximation $Df$ at each point in the domain of $f$, and since $f$ is Lipschitz, it is differentiable $\Hd^n$-a.e. so that $Jf$ exists $\Hd^n$-a.e.  The multiplicity function $\Hd^0(A\cap f^{-1}\{y\})$ takes into account the case in which $f$ is not 1-1 (see Figure \ref{areapic}).
 
\begin{figure} [H]
\begin{center}
    \scalebox{0.8}{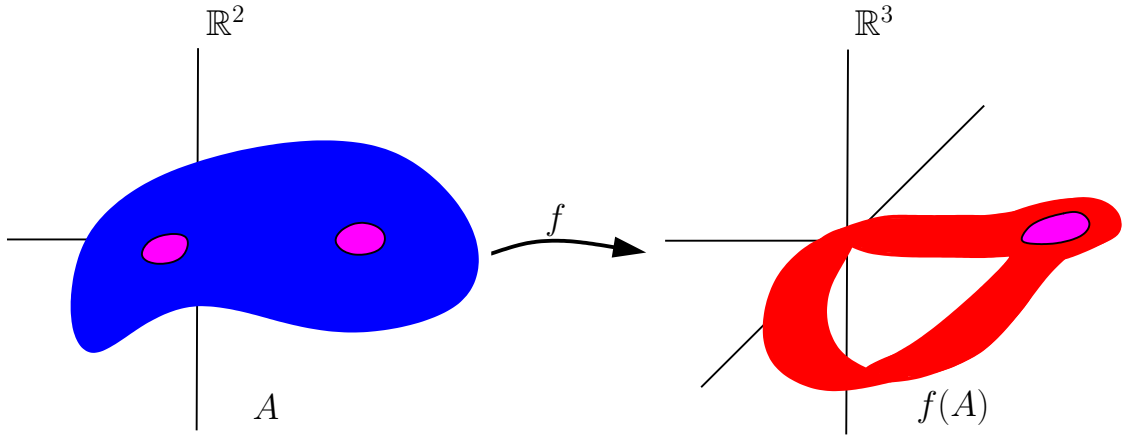}
      \caption[The Area Formula.]{Mapping $f$ such that the magenta areas correspond to multiplicity 2 regions in the image $f(A)$. Note that both cases show the domain of the set and not the graph space.} 
      \label{areapic}
 \end{center}
\end{figure}

Since Lipschitz maps cannot increase the Hausdorff dimension of $A$, we can disregard the case where $f$ is a space-filling curve. Clearly though, the $\Hd^n$-measure of $f(A)$ would be infinite, as a plane has infinite length, for instance. On the other hand, if $Df$ is singular, then $Jf$ is identically 0, so for example, in the case of a 2-d set being compressed into a 1-d set, we have information being lost, and the $\Hd^n$-measure of $f(A)$ is 0, as expected.

Applications of the Area Formula include the case of mapping a curve from $\R$ into $\R^3$ and finding its new length; computing the surface area of a graph or of a parametric hypersurface; or finding the volume of a set mapped into a submanifold. As a simple example, we present the case  of finding the length of a portion of the graph of the sine function in $\R$.

\begin{exm} Since $f=\sin(x)$ is Lipschitz, we let the collection of points $g(x)=(x,\sin(x)), x\in [0,2\pi]$ define the graph of $f$ over $[0,2\pi].$  We compute the arc length (i.e., the 1-d surface area) of $f$ on the set $[0,2\pi].$
 \[ Dg = 
 \begin{bmatrix}
1\\ 
\cos(x)
\end{bmatrix}
\]

Since $g$ is 1-1 and $Jg= \sqrt{1+\cos^2(x)}$, we have $\Hd^1(\text{Graph of f on }[0,2\pi])$= (arc length of $f$ on $[0,2\pi])=
\int_0^{2\pi}\sqrt{1+\cos^2(x)}\,dx \approx 7.64.$
\end{exm}

We next present a special case of the second part of the Area Formula stated above.

\begin{thm}\cite{morgan-2008-geometric}\label{areaII}
\textbf{(Area Formula II)} Let $f:\R^n \rightarrow \R^m$ be one-to-one and Lipschitz continuous, $n\leq m$, and $dx\equiv d\mathcal{L}^nx$. 
If $\Omega$ is an $\mathcal{L}^n$-measurable subset of $ \R^n$, then for each $\mathcal{L}^n\myell \Omega$-integrable function $g(x)$,
\begin{equation}\label{Area}
\int_{\Omega} g(x) Jf\,dx=\int_{f(\Omega)} g(f^{-1}(y))d\Hd^ny.
\end{equation}
\end{thm}

\begin{defn}\label{reg}
Let $f: \R^n\rightarrow \R^m$ be a differentiable function with $m\leq n$. A \textbf{regular value} of $f$ is a value $c\in\R^m$ such that the differential $Df$ is surjective at every preimage of $c$. This implies that $Df$ is full rank on the set $f^{-1}\{c\}$. 
\end{defn}

We are interested in here in the case $m=1$, which if $c$ is a regular value implies $||Df||\neq 0$ on $f^{-1}\{c\}$. In fact, if we let $f\in C^1$ and $f^{-1}\{c\}\subset B(0,R)$ (closed), we claim that there exists a neighborhood $E$ of $f^{-1}\{c\}$ on which $||Df||> 0$ for all $y\in E.$ Indeed, since $Df\in C^0$  we have $0<\epsilon<||D_xf||<M<\infty$ for all  $x\in f^{-1}\{c\}$. 

Let $W_{\frac{\epsilon}{2}}\equiv \{y \ \vert \ ||D_yf||<\frac{\epsilon}{2}\}.$ Now, suppose there exists a distinct sequence $\{y_i\}\subset W_{\frac{\epsilon}{2}}$ such that $d(y_i,f^{-1}\{c\})<\frac{1}{2^i}$, $i=1,2,\dots.$
Since $y_i\in B(0,2R)$, by the Bolzano-Weierstrass Theorem, there exists a subsequence $\{y_{i_k}\}$ and a point $y^*$ such that $y_{i_k}\rightarrow y^*$ as $k\rightarrow \infty$. Since $d(y^*,f^{-1}\{c\})=0$, we have $\{y_{i_k}\} \rightarrow f^{-1}\{c\}.$

But $||Df||$ is continuous, so this implies that $||D_{y^*}f||<\frac{\epsilon}{2}$, which is a contradiction. Thus there exists a $\delta>0$ such that $d(f^{-1}\{c\},W_{\frac{\epsilon}{2}})>\delta$, which implies $$\cup_{x\in f^{-1}\{c\}}(B(x,\delta)\cap W_{\frac{\epsilon}{2}})=\emptyset.$$ The result follows by taking $\cup_{x\in f^{-1}\{c\}}B(x,\delta)$ to be the neighborhood where $||Df||\neq0$.

\begin{rem}
The general case of showing $Df$ is full rank in a neighborhood of $f^{-1}\{c\}$ for $f$ as in Definition \ref{reg} can be shown using the continuity of the determinant of a non-zero square submatrix in $D_cf$ of maximal dimension, which exists because $D_cf$ is full rank.
\end{rem} 


\begin{thm}\label{regval}\cite{janich2013vector}
\textbf{Regular Value Theorem:} If $c\in M$ is a regular value of a differentiable map $f:N\rightarrow M$  (where $N$ and $M$ are $n$- and $m$- dimensional manifolds, respectively), then its preimage $f^{-1}\{c\}\subset N$ is a submanifold whose codimension is equal to the dimension of $M.$
\end{thm}

We note that $f^{-1}\{c\}$ is also a closed set since $f\in C(\R^n)$ implies that the inverse image of a closed set is closed.



\section{Measuring Level Sets of $C^{1,1}$ Functions}
The following theorem establishes a singular integral for measuring the $\mathcal{H}^{n-1}$ ``length'' of level sets of 
$C^{1,1}$ functions under suitable conditions.

\begin{thm}\label{first}
Let $f:\R^n\rightarrow \R^1$, and let $B(0,R)$ represent the closed ball of radius $R$ centered at the origin. If $0\in f(\R^n)$ is a regular value of $f\in C^{1,1}$ and $ f^{-1}\{0\} \subset B(0,R)$, then \
\begin{equation}\label{integral}
\lim_{k\rightarrow \infty}\left\{ \frac{1}{k}\int_{B(0,2R)}\left(\frac{||Df||}{|f|}\right)^{\frac{k-1}{k}}dx\right\}=2\mathcal{H}^{n-1}(\{x \ \vert \ f(x)=0\} ).
\end{equation}
\end{thm}

\begin{proof}

First, note that $||\cdot||$ is the operator norm, which in this case corresponds precisely to the Euclidean norm; and $|\cdot|$ is the Euclidean norm. Let $A\equiv \{x \ \vert \ f(x)=0\}$, the 0-level set of $f$, be such that $A\neq \emptyset$. Observe that Theorem \ref{regval} implies that $A$ is an $(n-1)$-dimensional submanifold, which implies that $\Hd^{n-1}(A)>0$. (In fact, $A$ is an embedded submanifold without boundary.) Since $A$ is closed, we have that $A$ is compact (as a subset of $\R^n$). Further, by Theorem \ref{imp}, $A$ is (locally) the graph of a $C^{1,1}$ function and thus has positive reach, which intuitively says that the normals to $A$ do not intersect in an $\epsilon$ neighborhood of $A$ for all $\epsilon<\reach(A)$. Now, recall the Tube Formula \cite{vixie-geo-deriv-tubes} for the $\Hd^n$-dimensional volume of the $\epsilon$-``tube'' $T_\epsilon$ surrounding $A$:
 \begin{equation}\label{tube}
\Hd^n(T_\epsilon)=2\epsilon \Hd^{n-1}(A)+o(\epsilon).
\end{equation}

\noindent Thus $T_\epsilon\subset B(0,R+\epsilon) \Rightarrow \Hd^n(T_\epsilon)<\infty,$ which together with (\ref{tube}) implies that $\Hd^{n-1}(A)<\infty$ as well.

Let $\epsilon>0$ and $L$ be the non-negative Lipschitz constant of $Df$ (see (\ref{lip}) below) such that $L\epsilon\ll ||D_zf||$ for all $z\in A$ (which we can do since $||Df|| >0$ on $A$). Further, let $||Df|| >0$ in the $\epsilon$-neighborhood $E$ of $A$, and suppose the normals to $A$ do not intersect in $E$ (due to the positive reach of $A$). Next observe that we only need to concern ourselves with integrating (\ref{integral}) over $E$ (where $||Df||\neq 0$!) since the $\frac{1}{k}$ in the limit sends everything outside of that neighborhood to 0. This follows from the fact that the integrand is continuous on the compact set $B(0,2R)\setminus E$.\\
\indent The idea that follows is that we want to integrate over $E$ so as to integrate first along the normals to $A$ and then integrate along $A$, yet due to the expansion (contraction) that occurs in dimensions $n\geq2$ as we push our 0-level set $A$ outward (inward) to form the neighborhood
$E$, which is captured in the Tube Formula (\ref{tube}), we cannot simply apply Fubini's Theorem to (\ref{integral}). Instead, we use the second version (\ref{Area}) of the Area Formula to account for this expansion in the following way.

 Let $C_\epsilon$ (which we will think of as  an $n$-dimensional manifold) be the ``cylinder''  formed from ``lifting'' $A$ up by $\epsilon$ in the $\R$ direction of the product space $\R^{n}\times \R$. We wish to apply Fubini's Theorem over $C_{\epsilon}$ (which we cannot do over $E$ due to the curvature of $A$.) So, we construct a diffeomorphism $H(w) : C_\epsilon \rightarrow E$   (quasi) explicitly using the normal map $n(z)$ of $A$, where $z\in A$. 
 
 Using $H(w)$ with the Area Formula, we then need to account for $JH$, the Jacobian of $H$, that acts as the ``expansion factor'' of the transformation $H(w)$; i.e. the amount of extra $n$-volume we get by expanding outward along the normals to $A$. But this we do by showing that as $\epsilon\rightarrow0, JH\rightarrow1.$ 
 

In our case, the Area Formula is given by
\begin{equation*}\label{ourArea}
 \int_{C_\epsilon} g(w) \;JH \;d\Hd^nw = \int_{H(C_\epsilon)} g(H^{-1}(y)) dy,
 \end{equation*}
where $w = (z,t)\in C_\epsilon \subset \R^n\times\R$, and $t\in[-\epsilon,\epsilon]$.

Now, setting
 \[g(H^{-1}(y)) =\frac{1}{k}\left(\frac{||Df(y)||}{|f(y)|}\right)^{\frac{k-1}{k}} \]
and $H(z,t) = \iota(z)+t\:n(\iota(z))$, where $\iota(z): \R^{n-1}\rightarrow \R^n$ is the inclusion map, we let $\iota(z)=z$ below, and we get that

\begin{equation*}\label{H}
 \int_{C_\epsilon} g(w) \;JH \;d\Hd^nw = \int_{C_\epsilon}
\frac{1}{k}\left(\frac{||Df(z+t\:n(z))||}{|f(z+t\:n(z))|}\right)^{\frac{k-1}{k}}\;JH
\;d\Hd^nw.
\end{equation*}

We can now use Fubini's on this, integrating, within
$C_{\epsilon}$, first in the $t$ direction and then in the $z$
direction (see Figure \ref{diffeo}).

\begin{figure} [H]
\begin{center}

      \scalebox{1.3}{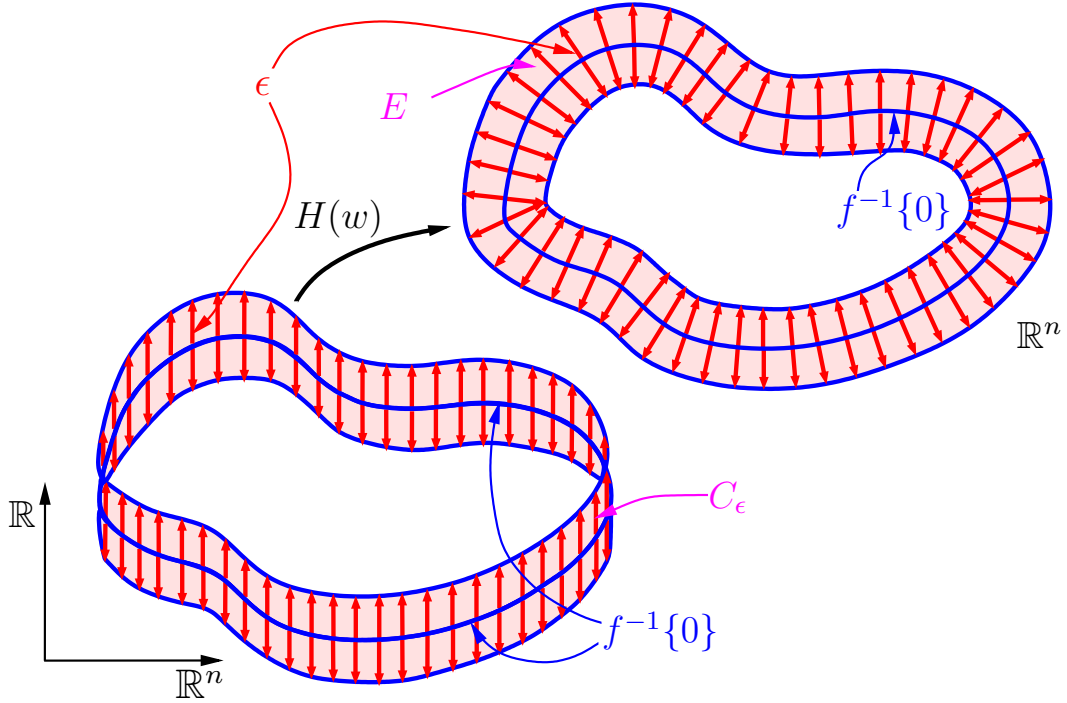}
      
      \caption{A diffeomorphism $H: C_\epsilon \rightarrow E.$} 
      \label{diffeo}
 \end{center}
\end{figure}

 
 Now, with $z\in A$, $y=z+t\:n(z)\in E$, and $t\in[-\epsilon,\epsilon]$ as above, we seek a bound on $||Df(z+t\:n(z))||=||D_y||$ as well as $f(y)$ in our neighborhood $E$. Here, we exploit the Lipschitz condition of $Df$, letting $0\leq L\equiv$ the Lipschitz constant of $Df.$ We recall that $f\in C^{1,1}$ implies that its derivative is Lipschitz, so we have, for all $x,v\in \R^n$,
 
 \begin{equation*}\label{lip}
  ||D_vf-D_xf||\leq L|v-x|.
  \end{equation*}
 
 \noindent From this, we find that
 
 \begin{equation*}
\left|\: ||D_yf||-||D_zf||\:\right|\leq ||D_yf-D_zf||\leq L|y-z|\leq L\epsilon,
\end{equation*}
 
  \noindent which says
 
 \begin{equation}\label{Df}
 ||D_zf||-L\epsilon\leq ||D_yf||=||Df(z+t\:n(z))||\leq ||D_zf||+L\epsilon.
 \end{equation}
 
 \noindent  It follows that $$(||D_zf||-L\epsilon)|t|\leq |f(z+t\:n(z))|\leq (||D_zf||+L\epsilon)|t|,$$ and thus

\begin{equation}\label{Df bound}
\frac{||D_zf||-L\epsilon}{(||D_zf||+L\epsilon)|t|}\leq\frac {||Df(z+t\:n(z))||}{|f(z+t\:n(z))|}\leq \frac{||D_zf||+L\epsilon}{(||D_zf||-L\epsilon)|t|}, \  \ \  \ \ L\epsilon \ll ||D_zf||.
\end{equation}

We note that there is a more geometric way to obtain the bound on $||D_yf||$ in (\ref{Df}) by looking at the ratio $\frac{h}{t}$, where $h$ is the increase (decrease) in ``rise'' that $f$ can obtain over the distance $t$ (so here we assume $t>0$) compared to the ``rise'' that $D_zf$ obtains over $t$, where $t$ is the length along the normal from $z$. This approach yields the same bounds as above (see Figure \ref{slope}). To see this,  we let $||D_yf_*||$ and $||D_yf^*||$ represent the minimum and maximum values $||Df||$ takes on $E$, respectively, and we look at the case where $0<||Df_*||\leq||D_yf||\leq||D_zf||$; the right hand side of (\ref{Df}) holds in a similar way if $||D_zf||\leq||D_yf||\leq||Df^*||<\infty$.

Notice that in Figure \ref{slope}, we see that $h<||D_zf||t-||Df_*||t$. The right side provides us with the greatest difference possible by assuming that $||Df||$ takes the value $||Df_*||$ at each point along the normal to $z$ between and including $z$ and $y$. This leads to the following inequality. 
$$\frac{h}{t}\leq|\:||D_zf||-||D_yf_*||\:|\leq||D_zf-D_yf_*||\leq L |z-y|\leq L\epsilon$$

And so, $$||D_zf||-L\epsilon\leq||D_zf|| - \frac{h}{t}=||D_yf||.$$

\begin{figure} [H]
\begin{center}

      \scalebox{2}{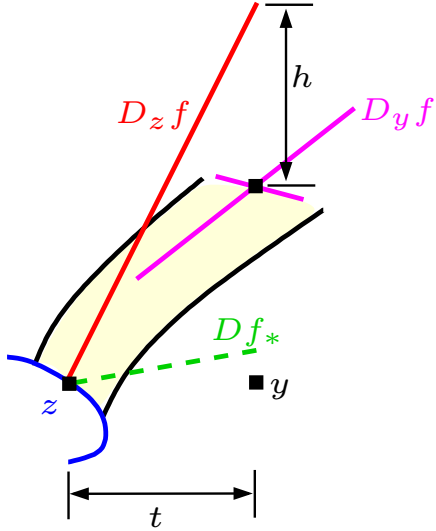}
      
      \caption{Bounding the derivative at $y.$} 
      \label{slope}
 \end{center}
\end{figure}

 \noindent Since $A$ has positive reach, we can show that  

\begin{equation}\label{JH bound}
(1-\epsilon \hat{\kappa})^{n-1} \leq JH \leq(1+\epsilon \hat{\kappa})^{n-1},
\end{equation}

 \noindent where $\hat{\kappa} \equiv \frac{1}{\text{reach}(A)}.$\\

 \noindent Now, putting this all together, we have

 \begin{align}
& \lim_{k\rightarrow \infty}\left\{ \frac{1}{k}\int_{B(0,2R)}\left(\frac{||Df||}{|f|}\right)^{\frac{k-1}{k}}dy\right\} 
\nonumber	\\
&=\lim_{k\rightarrow \infty}\left\{ \frac{1}{k}\int_{B(0,2R)\setminus E}\left(\frac{||Df||}{|f|}\right)^{\frac{k-1}{k}}dy\right\} \nonumber + \lim_{k\rightarrow \infty}\left\{ \frac{1}{k}\int_E\left(\frac{||Df||}{|f|}\right)^{\frac{k-1}{k}}dy\right\} \nonumber\\
& = \lim_{k\rightarrow \infty}\left\{ \frac{1}{k}\int_{{E}}\left(\frac{||Df(y)||}{|f(y)|}\right)^{\frac{k-1}{k}}dy\right\}\nonumber \\
& = \lim_{k\rightarrow \infty}\left\{ \frac{1}{k}\int_{C_{\epsilon}}\left(\frac{||Df(z+t\:n(z))||}{|f(z+t\:n(z))|}\right)^{\frac{k-1}{k}}JHd\Hd^nw\right\}\nonumber\\ 
& = \lim_{k\rightarrow \infty}\left\{ \frac{1}{k}\int_{A}\ \int_{-\epsilon}^\epsilon \left(\frac{||Df(z+t\:n(z))||}{|f(z+t\:n(z))|}\right)^{\frac{k-1}{k}}JHdt\,d\Hd^{n-1}z\right\}\nonumber\\
& \leq \lim_{k\rightarrow \infty}\left\{ \frac{1}{k}\int_{A}\int_{-\epsilon}^\epsilon \left(\frac{||D_zf||+L\epsilon}{(||D_zf||-L\epsilon)|t|}\right)^{\frac{k-1}{k}}(1+\epsilon \hat{\kappa})^{n-1}dt\,d\Hd^{n-1}z\right\} \ \ \ \ L\epsilon \ll ||D_zf||\nonumber\\
& = (1+\epsilon \hat{\kappa})^{n-1}\lim_{k\rightarrow \infty}\left\{ \frac{1}{k}\int_{A}\left(\frac{||D_zf||+L\epsilon}{||D_zf||-L\epsilon}\right)^{\frac{k-1}{k}}\int_{-\epsilon}^\epsilon \left(\frac{1}{|t|}\right)^{\frac{k-1}{k}}dt\,d\Hd^{n-1}z\right\} \nonumber\\
& = 2(1+\epsilon \hat{\kappa})^{n-1}\lim_{k\rightarrow \infty}\left\{t^{\frac{1}{k}}\bigg\rvert_{0}^\epsilon\int_{A}\left(\frac{||D_zf||+L\epsilon}{||D_zf||-L\epsilon}\right)^{\frac{k-1}{k}} d\Hd^{n-1}z\right\}  \label{dom}
\end{align}

\noindent Since, by the Dominated Convergence Theorem with $g=2\geq \left(\frac{||D_zf||+L\epsilon}{||D_zf||-L\epsilon}\right)^{\frac{k-1}{k}}$,

\[\lim_{k\rightarrow \infty}\left\{\int_{A}\left(\frac{||D_zf||+L\epsilon}{||D_zf||-L\epsilon}\right)^{\frac{k-1}{k}}d\Hd^{n-1}z\right\}
=\int_{A}\frac{||D_zf||+L\epsilon}{||D_zf||-L\epsilon}\,d\Hd^{n-1}z,\]
we split the limit in (\ref{dom}) to give the following:

\begin{align*}
& 2(1+\epsilon \hat{\kappa})^{n-1}\lim_{k\rightarrow \infty}\left\{t^{\frac{1}{k}}\bigg\rvert_{0}^\epsilon\int_{A}\left(\frac{||D_zf||+L\epsilon}{||D_zf||-L\epsilon}\right)^{\frac{k-1}{k}} d\Hd^{n-1}z\right\} \nonumber\ \ \ \ \ (\ref{dom})\\ 
& = 2(1+\epsilon \hat{\kappa})^{n-1}\lim_{k\rightarrow \infty}\left\{\epsilon^{\frac{1}{k}}\right\}\lim_{k\rightarrow \infty}\left\{\int_{A}\left(\frac{||D_zf||+L\epsilon}{||D_zf||-L\epsilon}\right)^{\frac{k-1}{k}} d\Hd^{n-1}z\right\} \nonumber\\
& = 2(1+\epsilon \hat{\kappa})^{n-1}\int_{A}\frac{||D_zf||+L\epsilon}{||D_zf||-L\epsilon}\,d\Hd^{n-1}z.
\end{align*}

\noindent We have thus far shown that 
\begin{equation}
 \lim_{k\rightarrow \infty}\left\{ \frac{1}{k}\int_{B(0,2R)}\left(\frac{||Df||}{|f|}\right)^{\frac{k-1}{k}}dy\right\} \label{uniform} \\
\leq 2(1+\epsilon \hat{\kappa})^{n-1}\int_{A}\frac{||D_zf||+L\epsilon}{||D_zf||-L\epsilon}\,d\Hd^{n-1}z. 
\end{equation}

But $\frac{||D_zf||+L\epsilon}{||D_zf||-L\epsilon}\rightarrow 1$ uniformly as $\epsilon\rightarrow0$ since setting $\epsilon_n=\frac{1}{n}$ and picking $\eta>0$, we find that $\left|\frac{||D_zf||+L/n}{||D_zf||-L/n}-1\right|=\left|\frac{2L}{n||D_zf||-L}\right|<\eta$ for all $z\in A$ for $n$ large enough. Letting $\epsilon\rightarrow0$ in (\ref{uniform}), we have

\begin{align}
& \lim_{k\rightarrow \infty}\left\{ \frac{1}{k}\int_{B(0,2R)}\left(\frac{||Df||}{|f|}\right)^{\frac{k-1}{k}}dy\right\}  \nonumber\\
&\leq2\lim_{\epsilon\rightarrow0}\left\{ (1+\epsilon \hat{\kappa})^{n-1}\int_{A}\frac{||D_zf||+L\epsilon}{||D_zf||-L\epsilon}\,d\Hd^{n-1}z\right\}\nonumber\\
&=2\lim_{\epsilon\rightarrow0}\left\{ (1+\epsilon \hat{\kappa})^{n-1}\right\}\lim_{\epsilon\rightarrow0}\left\{\int_{A}\frac{||D_zf||+L\epsilon}{||D_zf||-L\epsilon}\,d\Hd^{n-1}z\right\}\nonumber\\
&=2\int_{A}\lim_{\epsilon\rightarrow0}\frac{||D_zf||+L\epsilon}{||D_zf||-L\epsilon}\,d\Hd^{n-1}z\nonumber\\
& =2\int_{A} d\Hd^{n-1} \nonumber\\
&=2 \mathcal{H}^{n-1}(A).\label{result1}
\end{align}


An analogous argument with the left inequalities of (\ref{Df bound}) and (\ref{JH bound}) above yields the reverse inequality to (\ref{result1}), and we have 
\[\lim_{k\rightarrow \infty}\left\{ \frac{1}{k}\int_{B(0,2R)}\left(\frac{||Df||}{|f|}\right)^{\frac{k-1}{k}}dx\right\}=2\mathcal{H}^{n-1}(\{x \ \vert \ f(x)=0\} ).\]
\end{proof}

\begin{rem}
Under the following assumptions, we can measure any level set for $f$ in Theorem \ref{integral}. If $c\in f(\R^n)$ is a regular value and $ f^{-1}\{c\} \subset B(0,R)$, then we simply define $g\equiv f-c$ and use $g$ in place of $f$ in the theorem.

\end{rem}



\section{A Simple Example}

Let $f:\R^2\rightarrow\R$ be $f=x^2+y^2-1$, a paraboloid with 0-level set $E\equiv\{(x,y)\ \vert \ x^2+y^2=1\}$, i.e. the unit circle. We know that twice the length of this set is $4\pi$, so we will show that the integral above indeed yields this result (see Figure \ref{example2}). 
\begin{figure}[H]
\centering
\includegraphics[width=12cm]{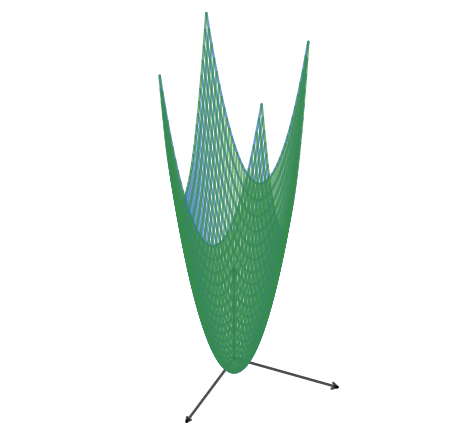}
\caption[The graph of $f=x^2+y^2-1.$]{The graph of $f=x^2+y^2-1$ with 0-level set the unit circle.}
\label{example2}
\end{figure}

For functions mapping from $\R^n\rightarrow \R$, $||Df||$ is simply the vector norm of the gradient since the operator norm measures the greatest expansion of any unit vector that occurs under $Df$, which occurs exactly in the direction of the gradient. Thus, $\nabla f = (2x,2y),$ and $|\nabla f|=2\sqrt{x^2+y^2}.$ If we consider $B(0,1)$ to be closed, then we shall set $B(0,2R)=B(0,2).$

 \begin{align}
 \lim_{k\rightarrow \infty}\left\{ \frac{1}{k}\int_{B(0,2R)}\left(\frac{||Df||}{|f|}\right)^{\frac{k-1}{k}}d\vec{x}\right\} &=2\lim_{k\rightarrow \infty}\left\{ \frac{1}{k}\int_{B(0,2)}\left(\frac{\sqrt{x^2+y^2}}{x^2+y^2-1}\right)^{\frac{k-1}{k}}dxdy\right\} \nonumber \\
 &=2\lim_{k\rightarrow \infty}\left\{ \frac{1}{k}\int_{B(0,1)}\left(\frac{\sqrt{x^2+y^2}}{1-x^2-y^2}\right)^{\frac{k-1}{k}}dxdy\right\} \nonumber \\
 &+2\lim_{k\rightarrow \infty}\left\{ \frac{1}{k}\int_{B(0,2)\setminus B(0,1)}\left(\frac{\sqrt{x^2+y^2}}{x^2+y^2-1}\right)^{\frac{k-1}{k}}dxdy\right\} \nonumber \\
  &=2\lim_{k\rightarrow \infty}\left\{ \frac{1}{k}\int_{0}^{2\pi}\int_{0}^{1}\left(\frac{r}{1-r^2}\right)^{\frac{k-1}{k}}rdrd\theta\right\} \nonumber \\
 &+2\lim_{k\rightarrow \infty}\left\{ \frac{1}{k}\int_{0}^{2\pi}\int_{1}^{2}\left(\frac{r}{r^2-1}\right)^{\frac{k-1}{k}}rdrd\theta\right\} \label{exam}
\end{align}

Using numerical integration techniques, we find the value of (\ref{exam}) rapidly converging to $4\pi\approx 12.566$ as predicted. For example, with $k=1000$, we have 
 \begin{align*}
  &2\lim_{k\rightarrow \infty}\left\{ \frac{1}{1000}\int_{0}^{2\pi}\int_{0}^{1}\left(\frac{r}{1-r^2}\right)^{\frac{999}{1000}}rdrd\theta\right\}\nonumber\\
+&2\lim_{k\rightarrow \infty}\left\{ \frac{1}{1000}\int_{0}^{2\pi}\int_{1}^{2}\left(\frac{r}{r^2-1}\right)^{\frac{999}{1000}}rdrd\theta\right\}\nonumber\\
&\approx 12.568.
 \end{align*}

For an analytic solution, Mathematica\textsuperscript{\textregistered} provides the following solution to (\ref{exam}).
 \begin{align*}
&2\lim_{k\rightarrow \infty}\left\{ \frac{1}{k}\int_{0}^{2\pi}\int_{0}^{1}\left(\frac{r}{1-r^2}\right)^{\frac{k-1}{k}}rdrd\theta\right\} 
 +2\lim_{k\rightarrow \infty}\left\{ \frac{1}{k}\int_{0}^{2\pi}\int_{1}^{2}\left(\frac{r}{r^2-1}\right)^{\frac{k-1}{k}}rdrd\theta\right\} \label{unit}\nonumber \\
 &=2\pi\lim_{k\rightarrow \infty}\left\{\frac{1}{k} \ \frac{2^\frac{-1}{k}\Gamma\left(\frac{3}{2}-\frac{1}{2k}\right)\Gamma\left(\frac{1}{k}\right)}{\Gamma\left(\frac{1}{2}\left(3+\frac{1}{k}\right)\right)}\right\}\nonumber\\
 &+2\pi\lim_{k\rightarrow \infty} \left\{\frac{1}{k} \ \Gamma\left(-\frac{k+1}{2k}\right) \left[ \frac{2^\frac{-1}{k}\Gamma\left(\frac{1}{k}\right)}{\Gamma\left(\frac{1}{2}\left(\frac{1}{k}-1\right)\right)}-2 \, {}_2\tilde{F}_1\left(\frac{k-1}{k}, \, -\frac{k+1}{2k}; \, \frac{k-1}{2k};\,\frac{1}{4} \right)\right]\right\}\nonumber\\
  &=2\pi\lim_{k\rightarrow \infty}\left\{\frac{1}{k} \ 2^\frac{-1}{k}\Gamma\left(\frac{1}{k}\right)\right\}\nonumber\\
 &+2\pi\lim_{k\rightarrow \infty} \left\{\frac{1}{k} \  \left[ 2^\frac{-1}{k}\Gamma\left(\frac{1}{k}\right)-2\Gamma\left(\frac{-1}{2}\right) \, {}_2\tilde{F}_1\left(1, \, \frac{-1}{2}; \, \frac{1}{2};\,\frac{1}{4} \right)\right]\right\}\nonumber\\
 &=2\pi\lim_{k\rightarrow \infty}\left\{ 1+1-\frac{2}{k}\Gamma\left(\frac{-1}{2}\right)\frac{-2+\text{coth}^{-1}(2)}{2\sqrt{\pi}}\right\}\nonumber\\
  &=2\pi\lim_{k\rightarrow \infty}\left\{ 1+1-\frac{2}{k}(-2\sqrt{\pi})(.409\dots)\right\}\nonumber\\
 &=4\pi
 \end{align*}

We note that the regularized hypergeometric function $${}_2\tilde{F}_1(b,a;c;z)\equiv\frac{{}_2F_1(a,b;c;z)}{\sqrt{\pi}},$$ where ${}_2F_1(a,b;c;z)$ is the hypergeometric function, which in turn is defined by the power series $$\sum_{n=0}^\infty\frac{(a)_n(b)_n}{(c)_n}\frac{z^n}{n!}.$$ Here $|z|<1$ and $(q)_n$ is the rising Pochhammer symbol defined by 
\[(q)_n = \left\{ \begin{array}{rcl}
 1 &n=0 \\
q(q+1)\cdots(q+n-1) & n>0. \end{array}\right. \]

\noindent Further, note above that $\lim_{k\rightarrow \infty}\left\{\frac{1}{k} \Gamma\left(\frac{1}{k}\right)\right\}=1.$


\section{A $C^1$ Boundary Measure for $n=2$}
We next present the second main result, that of computing the $\mathcal{H}^{1}$-measure of an embedded $C^1$ submanifold in $\R^2$ using the singular integral in (\ref{integral}) above. 

\begin{thm}\label{C1}
Let  $f(x)=\inf\{|x-y| \ \vert \ y\in A\}:\R^2\rightarrow \R$ be the \emph{\textbf{distance function}} to $A\subset\R^2$.  Further, suppose $A$ is a $C^1$ closed embedded 1-dimensional submanifold of $\R^2$ such that $A\subset B(0,R)$ (closed) and $\mathcal{H}^{1}(A)<\infty$. Then we have
\begin{equation}\label{int}
\lim_{k\rightarrow \infty}\left\{ \frac{1}{k}\int_{B(0,2R)}\left(\frac{|Df|}{|f|}\right)^{\frac{k-1}{k}}dy\right\}=2\mathcal{H}^{1}(A).
\end{equation}
\end{thm}

\begin{proof}
 As $A$ is a closed submanifold, i.e. it is compact without boundary, we assume $A\neq \emptyset,$ which implies that $\mathcal{H}^{1}(A)>0.$ We also note that $|Df|\equiv 1 \ \Hd^2$-a.e.; $Df$ is undefined on $A$ itself; and that $A$ being $C^1$  means that $A$ is locally the graph of a $C^1$ function $g:\R\rightarrow \R$. 


\begin{enumerate}
\item
\begin{enumerate}
\item
Since $A$ is an embedded submanifold, then globally speaking, no matter how close one portion of $A$ comes to another portion, there is always at least a distance $\eta^*>0$ between them. We will chose $\eta$ below in such a way that $\eta\le \eta^*/4$ so that we may cover $A$ with rectangles extending outward $\eta$ on both sides without overlapping (see Figure \ref{global}).\\

\begin{figure} [H]
\centering
      \scalebox{.8}{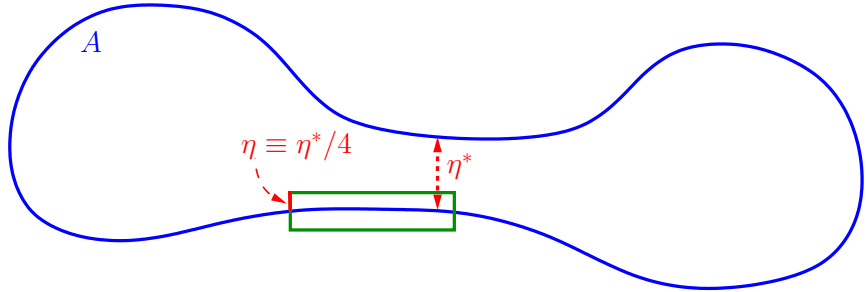}
     
      \caption{\label{global}Global distance of $A$ to itself. } 
\end{figure}

\item

Since $A$ is a compact submanifold of $\R^2$, it is a compact set in $\R^2$, hence we can cover $A$ with a finite number of closed disks $B(x_i,\delta)$ with $x_i\in A$ for $i=1,2,\dots,M(\delta),$ where $\delta$ is the fineness of our cover and $M(\delta)$ is the number of disks. Since $A$ is $C^1$, for each $x_i$, we define $T_i^{}\equiv T_{x_i}A\cap B(x_i,\delta)$, the portion of the linear approximation to $A$ at the point $x_i$ contained in each ball $B(x_i,\delta)$, which is naturally a diameter of $B(x_i,\delta)$ and thus has length $2\delta$ (see Figure \ref{disks}). 

\begin{figure} [H]
\centering
      \scalebox{1.1}{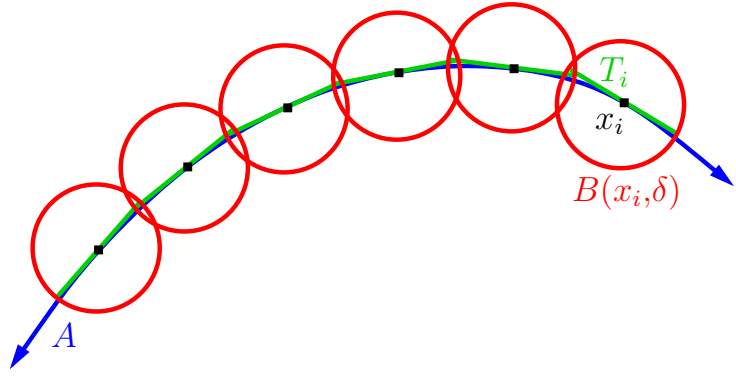}
     
      \caption{\label{disks}Covering $A$ with disks $B(x_i,\delta)$.} 
\end{figure}

\item

Using the fact that $f$ is continuous on the compact set $B(0,2R)$  and that $|Df|\equiv 1$ $\Hd^2$-a.e., we see that the integrand in (\ref{int}) is bounded outside of a neighborhood $E$ of $A$ contained in $B(0,2R) $ (see Figure \ref{ball}). Thus we have
\begin{equation}\label{E}
\lim_{k\rightarrow \infty}\left\{ \frac{1}{k}\int_{B(0,2R)}\left(\frac{|Df|}{|f|}\right)^{\frac{k-1}{k}}dy\right\}= \lim_{k\rightarrow \infty}\left\{ \frac{1}{k}\int_{E}\left(\frac{|Df|}{|f|}\right)^{\frac{k-1}{k}}dy\right\}.
\end{equation}
\begin{figure} [H]
\centering
      \scalebox{.7}{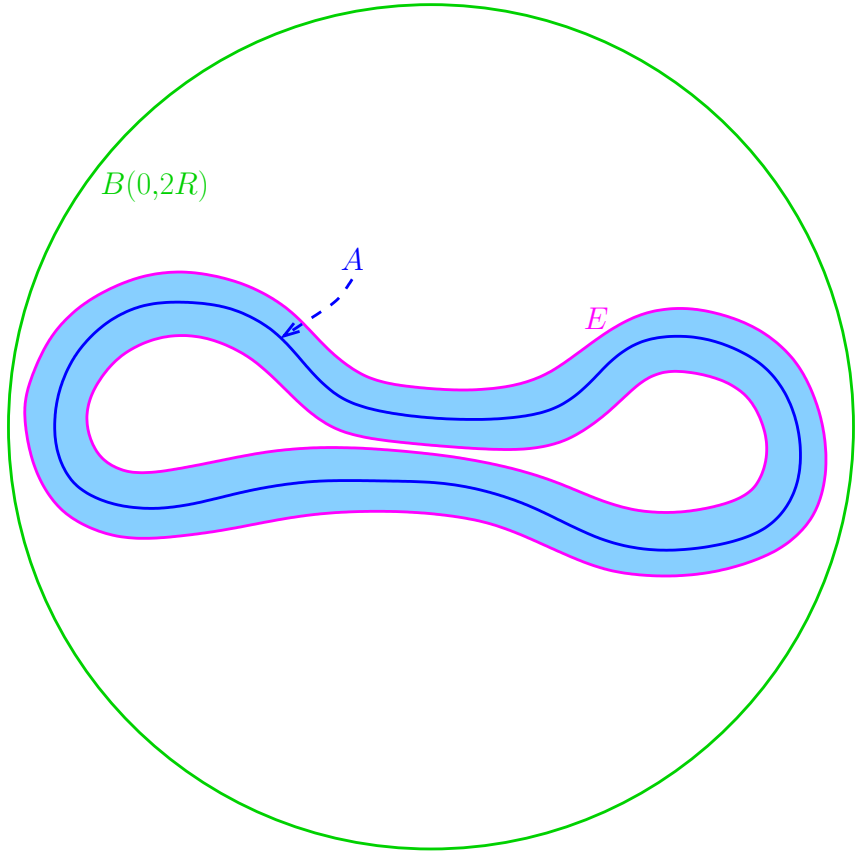}
     
      \caption[Taking the limit inside $B(0,2R)$]{\label{ball}The $1/k$ sends everything outside $E$ to 0 in the limit. } 
\end{figure}
\end{enumerate}

\item
\begin{enumerate}
\item
Let $E_i^\eta$ be a rectangle that is centered on $T_i$ with height $2\eta$ and width $2\delta$ for $i=1,2,\dots,M(\delta)$ and with $\eta$ chosen as above. We will integrate in (\ref{int}) over each $E_i^\eta$ and sum the results so that we can use the normals $\vec{N}$ to each $T_i$ to approximate the normals $\vec{n}$ to $A\cap E_i^\eta$. However, we need to bound how much ``overlap'' there is between the $E_i^\eta$'s so that the integral over the collection $\{E_i^\eta\}_{i=1}^{M(\delta)}$ converges to the integral over  $\cup_{i=1}^{M(\delta)} E_i^\eta$ as $\delta\rightarrow0.$ We do this in the following way, noting that since $A\subset\cup_{i=1}^{M(\delta)} E_i^\eta\subset B(0,2R)$, we can apply (\ref{E}) above (see Figure \ref{rectangles}).

\begin{figure} [H]
\centering
      \scalebox{1.3}{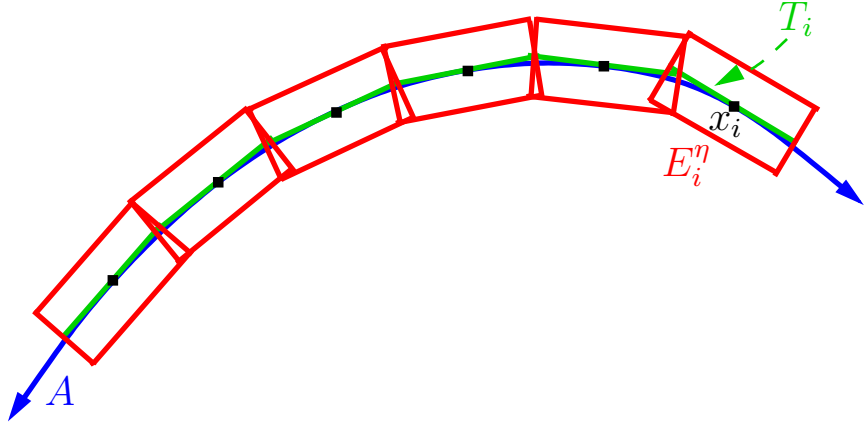}
     
      \caption{\label{rectangles}Covering $A$ with rectangles $E^\eta_i$. } 
\end{figure}

\item Let $w,z\in A$ and  $\alpha>0$. By the \textit{uniform continuity} of $Dg$, (locally) we can choose $\delta$ small enough so that 
\begin{equation*}\label{cont1}
|w-z|<\delta \Rightarrow |D_wg-D_{z}g|<\alpha.
\end{equation*}

 \noindent But this implies that
\begin{equation*}\label{cont2}
|g(w)-g(z)|<\alpha\delta\equiv\epsilon.
\end{equation*}

We pick $\delta$ small enough so that letting $\eta\equiv\sqrt{\alpha}\delta,$ we still have $\eta\le \eta^*/4$ from above. Then

\begin{equation*}\label{ratio0}
\frac{\eta}{\epsilon}=\frac{\sqrt{\alpha}\delta}{\alpha\delta}=\frac{1}{\sqrt{\alpha}}\rightarrow \infty
\end{equation*}
as $\alpha,\delta\rightarrow0$.

Similarly,

\begin{equation}\label{ratio2}
\frac{\delta}{\eta}=\frac{\delta}{\sqrt{\alpha}\delta}=\frac{1}{\sqrt{\alpha}}\rightarrow \infty
\end{equation}
as $\alpha,\delta \rightarrow0$.

 \noindent Lastly,  we have
\begin{equation}\label{ratio1}
\frac{\delta}{\epsilon}=\frac{\delta}{\alpha\delta}=\frac{1}{\alpha}\rightarrow \infty
\end{equation}
as  $\alpha,\delta \rightarrow0$.

By the geometric definition of the derivative \cite{vixie-geometric-derivative}, we observe that $A\cap E_i^\eta$ is contained in a cone $K_i^\epsilon$ centered on $x_i$ and $T_i$ of max height 2$\epsilon$ (at a distance $\delta$ from $x_i$ along $T_i$) and angular width $2\theta$ (see Figure \ref{cone1}).

\end{enumerate}

\begin{figure} [H]
\centering
      \scalebox{.8}{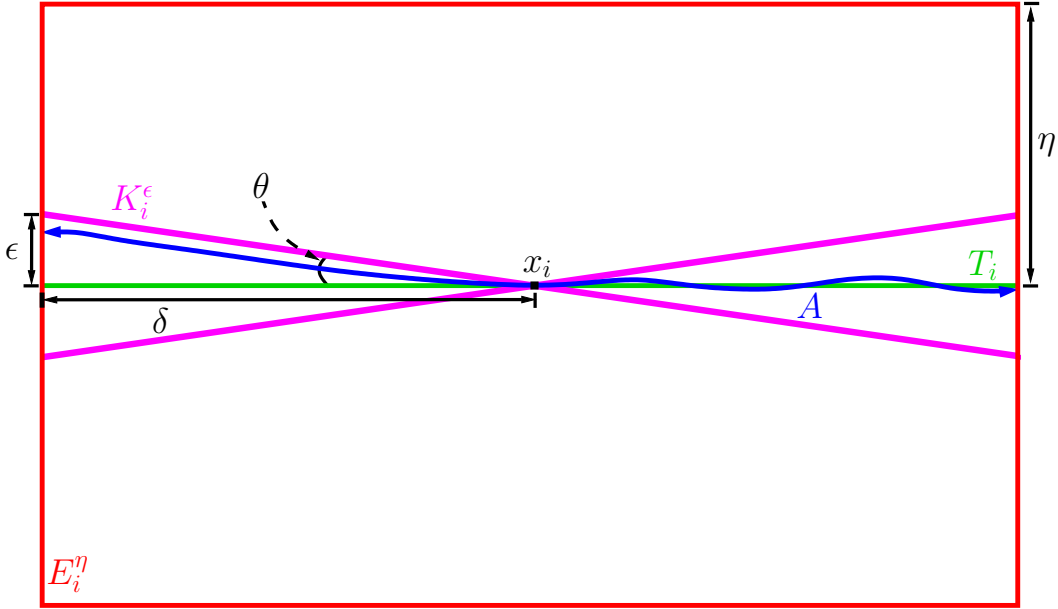}
     
      \caption[$A\cap E^\eta_i$ is contained in $K^\epsilon_i$]{\label{cone1}For each $x_i, A\cap E^\eta_i$ is contained the cone $K^\epsilon_i$. Note that the image is not drawn to scale since $\delta \gg\eta \gg\epsilon.$ } 
\end{figure}
\item
\begin{enumerate}
\item

We now specify the  manner in which we cover $A$ with disks, and thus with $E_i^\eta$'s. We wish for the ``overlaps'' of the $E_i^\eta$'s to contain at least an $\eta/2$-neighborhood of $A$, and we want 

 $$\Hd^2(E_i^\eta\cap E_j^\eta)\le (2\eta)^2=4\eta^2, \ \ i\neq j.$$
 
 By choosing $\delta$ small enough, we can arrange each $E_i^\eta$ so that it intersects with its neighbors only in the $2\eta$-squares on each end of each rectangle $E_i^\eta$ (see Figure \ref{intersect}).


More precisely, 
\begin{equation*}\label{ratio4}
\Hd^2(E_i^\eta)=2\delta2\eta \gg 4\eta^2\geq \Hd^2(E_i^\eta\cap E_j^\eta), \ \ i\neq j,
\end{equation*}
since (\ref{ratio2}) above implies 
\begin{equation*}\label{ratio4}
\delta \gg \eta.
\end{equation*}
Thus sending $\delta$ to 0 will send the total area of the pairwise ``overlap'' of the $E_i^\eta$'s to 0 faster than the area of the collection $\{E_i^\eta\}_{i=1}^{M(\delta)}$ goes to 0.

\begin{figure} [H]
\centering
      \scalebox{.7}{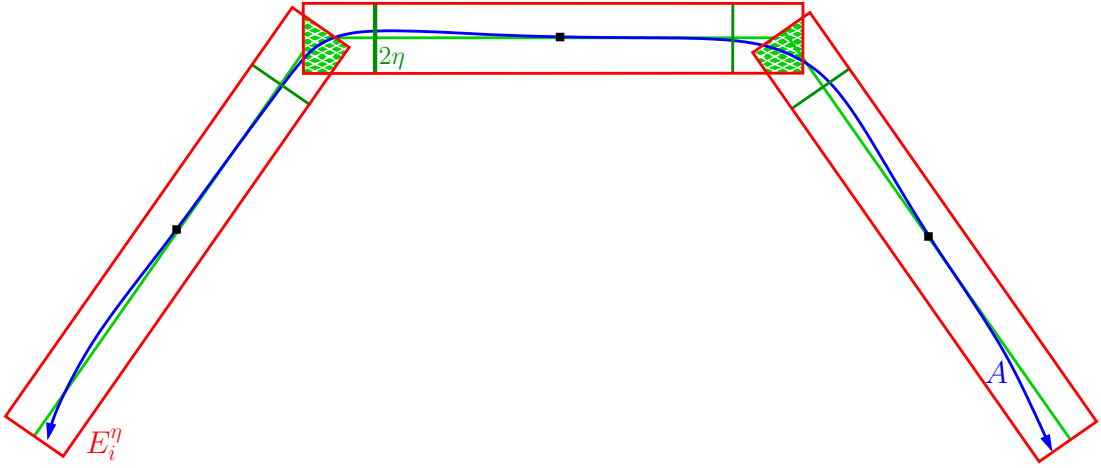}
     
      \caption{\label{intersect}Overlaying $A$ with $E^\eta_i$'s} 
\end{figure}

\item  We set 
$$ I =  \frac{1}{k}\int_{\Omega}\left(\frac{|Df|}{|f|}\right)^{\frac{k-1}{k}}dz,$$
where $\Omega\equiv \{E_i^\eta\cap E_{j}^\eta\}_{i\neq j},$ the collection of ``overlap'' of all the $E_i^\eta$'s taken pairwise.  Now we can write

\begin{align*}\label{I}
& \lim_{k\rightarrow \infty}\left\{ \frac{1}{k}\int_{\cup_{i=1}^{M(\delta)} E_i^\eta}\left(\frac{|Df|}{|f|}\right)^{\frac{k-1}{k}}dz\right\}\nonumber\\
&= \lim_{k\rightarrow \infty}\left\{ \frac{1}{k}\int_{\{E_i^\eta\}_{i=1}^{M(\delta)}}\left(\frac{|Df|}{|f|}\right)^{\frac{k-1}{k}}dz-I\right\} \nonumber\\
&=\lim_{k\rightarrow \infty}\left\{ \frac{1-\lambda}{k}\int_{\{E_i^\eta\}_{i=1}^{M(\delta)}}\left(\frac{|Df|}{|f|}\right)^{\frac{k-1}{k}}dz\right\} \nonumber\\
\end{align*}

since $I$ is some  fraction $0\leq\lambda\equiv\lambda(\delta)<1$ of the same integral  taken over the whole (finite) collection $\{E^\eta_i\}_{i=1}^{M(\delta)}$.   Based on the argument just given, $\lambda\rightarrow0$ as $\delta\rightarrow 0$.



\end{enumerate}
\item
\begin{enumerate}
\item We now wish to find bounds on $f$ by  showing that the normals to $A\cap E_i^\eta$ converge to the normals to  $T_i$ as $\delta\rightarrow0$. The key lemma here is as follows:

\begin{lem}

Let $\epsilon> \epsilon^*>0$ and let $K^\epsilon_i$ be the cone of angular width $2\theta$ centered on $x\in T_i$ containing $A\cap E_i^\eta$. There exists a cone $C^{\epsilon^*}_i$ such that we may translate it (without tipping)  along $A\cap E_i^\eta$ so that $A$ is contained inside $C^{\epsilon^*}_i$ for all points of $A$ within $\delta^*<\delta$ of its center point (see Figure \ref{cone2}).
\end{lem}


\begin{proof}
Let $\delta>0$ be as above and $w,z\in A\cap E^\epsilon_i$. Then from the uniform continuity of $Dg$, we find that $|w-z|<\delta$ implies $\angle (D_wg,D_zg)\le 2\theta$, that is the angle between $D_wg$ and $D_zg$ is bounded by $2\theta$, which defines  the angular width of $K^\epsilon_i$.     Thus the normals $\vec{n}$ to $A\cap E^\epsilon_i$ can only deviate from the normals $\vec{N}$ to $T_i$ by at most $\theta=\tan^{-1}(\frac{\epsilon}{\delta})$.  

Now, let $a\in A\cap E^\epsilon_i$ and let $\vec{N}(x)$ be the normal to $T_i$ containing $a$. Pick a point $z=(x,t) \in \vec{N}(x)$, where $t\le \epsilon$ and draw the line segments through $z$ on each side of $\vec{N}(x)$ forming an angle $\theta$ with $\vec{N}(x)$. Without loss of generality, pick one of these segments and label its intersection with $A$ as  $\oldhat{a}\in A\cap E^\epsilon_i$. We find, from the above argument,  that $\vec{n}(\oldhat{a})$, the normal to $\oldhat{a}$, cannot deviate from its corresponding $\vec{N}(\oldhat{x})$ by more than $\theta$, and thus $\angle (D_ag,D_{\oldhat{a}}g)\le 2\theta$. But this implies that $\oldhat{a}$ must be inside the cone $C^{\epsilon^*}_i$ centered on $a$ whose center line $l$ is parallel to $T_i$.

Letting $z=(x,t)$ be any point along $\vec{N}(x)$ with $t\le \epsilon$, we find that the corresponding $\oldhat{a}\in C^{\epsilon^*}_i$, which proves the assertion. We observe that from \eqref{ratio1}, as $\delta\rightarrow0, \theta$ also goes to 0.

\end{proof}

\item

As shown in  Figure \ref{cone2}, using our lemma we wish to bound $f(z)$ for $z=(x,t)\in E_i^\eta$ such that $x\in T_i$ and $t\in[-\eta,\eta]$ is the distance to $x$ in the direction of the normal $\vec{N}(x)$ to $T_i$. But if we translate $C^{\epsilon^*}_i$ to each $z\in A\cap E_i^\eta$ (keeping its center line $l$ parallel to $T_i$), we find (with $0\le\theta<\pi/2$)
 
\begin{equation*}\label{boundf}
 \cos(\theta)|t^*|\leq|f(z^*)|\leq |t^*|,
\end{equation*}



\noindent which implies

\begin{equation}\label{bound}
\frac{1}{|t^*|}\leq \frac{1}{|f(z^*)|}\leq \frac{1}{\cos(\theta)|t^*|}.
\end{equation}

Here, $z^*=(x,t^*)$, where $t^*$ represents the distance from $z=(x,t)$ to $a\in A$ along the normal $\vec{N}(x)$ instead of the distance to $x\in T_i$. The actual distance $t$ to $T_i$ is either $t^*+ \beta$ or  $t^*- \beta$, where $0\le \beta\le \epsilon$, depending on what side of $T_i$ 
$a$ is and since $A$ is contained in $K^\epsilon_i$.

\begin{figure} [H]
\centering
      \scalebox{1.3}{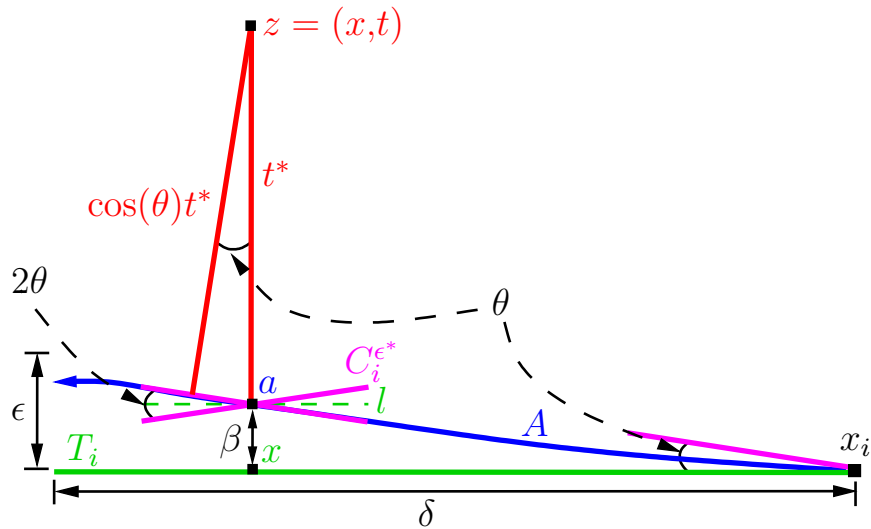}
     
      \caption[Translating $C^{\epsilon^*}_i$ along $A$]{\label{cone2}Translating $C^{\epsilon^*}_i$ along $A$ so that line $l$ is always parallel to $T_i$ and  $t^*$ is the distance from $z$ to $a$.} 
\end{figure}

Now, we wish to integrate $1/|f(z)|$ over $[-\eta,\eta]$, with $\eta$ being measured from $T_i$, but to use our bound in (\ref{bound}), we make the substitution $z^*=(x,t)\pm(0,\beta_x)$, with $t^*=t\pm\beta_x$, so that $dz^*=dz$ and $dt^*=dt$. Note that the points $(x,t)$ and $(x,t^*)$ represent the same point in $\R^2$, but when integrating, we will have different bounds in the integral since our datum is different if our point $a\in A$ is not also in $T_i$, that is if $a\neq x$. Therefore, in the case when $a$ is above $T_i$, we let our bounds be $\hat{\eta}^+_x=\eta-\beta_x$ and $\hat{\eta}^-_x=-(\eta+\beta_x),$ where $\beta_x$ depends on $x\in T_i$. In the case when $a$ is below $T_i$, we have $\hat{\eta}^+_x=\eta+\beta_x$ and $\hat{\eta}^-_x=-(\eta-\beta_x)$ (see Figure \ref{bounds}).

\begin{figure} [H]
\centering
      \scalebox{1.3}{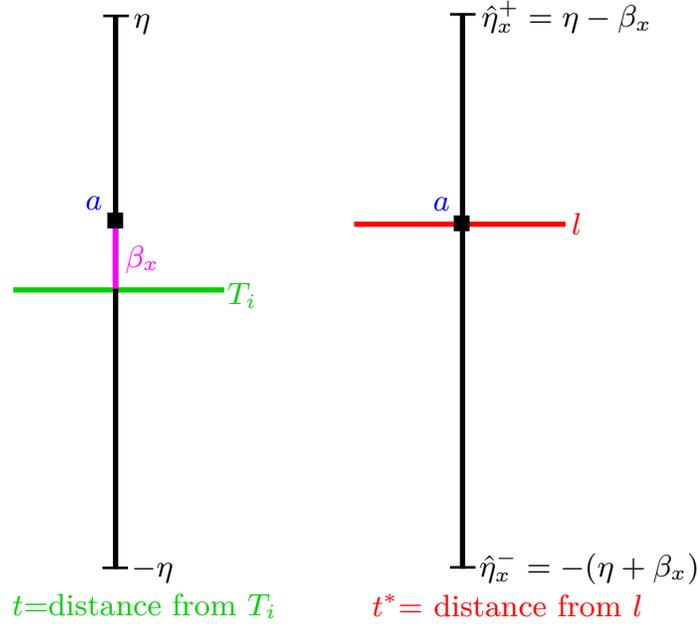}
     
      \caption[Bounds of integration]{\label{bounds}On the left, we integrate $1/|f(z)|=1/|f(x,t)|$ over $[-\eta,\eta]$ with $t$ measured as the distance from $x\in T_i$. On the right, we integrate $1/|f(z^*)|=1/|f(x,t^*)|$ over $[\hat{\eta}^-_x,\hat{\eta}^+_x]$ with $t^*$ measured as the distance from $a\in l$. Both are for the case where $a$ is above $ T_i$.} 
\end{figure}

In either case, we are integrating over the same interval. Since all terms in (\ref{bound}) are positive, this will be greater than integrating over $t\in[-(\eta-\epsilon),\eta-\epsilon]$ and less than integrating over $t\in[-(\eta+\epsilon),\eta+\epsilon]$, and we have

\begin{align*}
\int_{-\eta}^{\eta}\left(\frac{1}{|f(z)|}\right)^{\frac{k-1}{k}}dz&=\int_{\hat{\eta}^-_x}^{\hat{\eta}^+_x}\left(\frac{1}{|f(z^*)|}\right)^{\frac{k-1}{k}}dz^* \nonumber\\
&\le\int_{\hat{\eta}^-_x}^{\hat{\eta}^+_x}\left(\frac{1}{\cos(\theta)|t^*|}\right)^{\frac{k-1}{k}}dt^* \nonumber\\
& \le \int_{-(\eta+\epsilon)}^{\eta+\epsilon}\left(\frac{1}{\cos(\theta)|t|}\right)^{\frac{k-1}{k}}dt,
\end{align*}

where $k=1,2,3,\dots$. Similarly, 

\begin{align}
\int_{-\eta}^{\eta}\left(\frac{1}{|f(z)|}\right)^{\frac{k-1}{k}}dz&=\int_{\hat{\eta}^-_x}^{\hat{\eta}^+_x}\left(\frac{1}{|f(z^*)|}\right)^{\frac{k-1}{k}}dz^* \nonumber\\
&\ge\int_{\hat{\eta}^-_x}^{\hat{\eta}^+_x}\left(\frac{1}{|t^*|}\right)^{\frac{k-1}{k}}dt^* \nonumber\\
& \ge \int_{-(\eta-\epsilon)}^{\eta-\epsilon}\left(\frac{1}{|t|}\right)^{\frac{k-1}{k}}dt.\label{etaxlower}
\end{align}

\end{enumerate}

\item
Putting this all together, we now have that 
\begin{align}
&\lim_{k\rightarrow \infty}\left\{ \frac{1}{k}\int_{B(0,2R)}\left(\frac{|Df|}{|f|}\right)^{\frac{k-1}{k}}dz\right\}\nonumber\\
&= \lim_{k\rightarrow \infty}\left\{ \frac{1}{k}\int_{\cup_{i=1}^{M(\delta)} E_i^\eta}\left(\frac{|Df|}{|f|}\right)^{\frac{k-1}{k}}dz\right\}\nonumber\\
&= \lim_{k\rightarrow \infty}\left\{ \frac{1}{k}\int_{\{E_i^\eta\}_{i=1}^{M(\delta)}}\left(\frac{|Df|}{|f|}\right)^{\frac{k-1}{k}}dz-I\right\} \nonumber\\
&=\lim_{k\rightarrow \infty}\left\{ \frac{1-\lambda}{k}\sum_{i=1}^{M(\delta)}\int_{E_i^\eta}\left(\frac{1}{|f|}\right)^{\frac{k-1}{k}}dz\right\}\nonumber\\
&=\lim_{k\rightarrow \infty}\left\{ \frac{1-\lambda}{k}\sum_{i=1}^{M(\delta)}\int_{T_i}\int_{-\eta}^\eta\left(\frac{1}{|f|}\right)^{\frac{k-1}{k}}dt\,d\Hd^{1}(x)\right\}\nonumber\\
&\leq \lim_{k\rightarrow \infty}\left\{ \frac{1-\lambda}{k}\sum_{i=1}^{M(\delta)}\int_{T_i}\int_{-(\eta+\epsilon)}^{\eta+\epsilon}\left(\frac{1}{\cos(\theta) |t|}\right)^{\frac{k-1}{k}}dt\,d\Hd^{1}(x)\right\}\nonumber \\
&=\lim_{k\rightarrow \infty}\left\{ \frac{1-\lambda}{k}\int_{-(\eta+\epsilon)}^{\eta+\epsilon}\left(\frac{1}{|t|}\right)^{\frac{k-1}{k}}dt\left(\frac{1}{\cos(\theta) }\right)^{\frac{k-1}{k}}\sum_{i=1}^{M(\delta)}\int_{T_i}d\Hd^{1}(x)\right\}\nonumber\\
&=\lim_{k\rightarrow \infty}\left\{ \frac{1-\lambda}{k}\int_{-(\eta+\epsilon)}^{\eta+\epsilon}\left(\frac{1}{|t|}\right)^{\frac{k-1}{k}}dt\right\}\lim_{k\rightarrow \infty}\left\{\left(\frac{1}{\cos(\theta) }\right)^{\frac{k-1}{k}}\sum_{i=1}^{M(\delta)}\Hd^{1}(T_i)\right\}\nonumber\\
&=2(1-\lambda)\lim_{k\rightarrow \infty}\left\{ (\eta+\epsilon)^{\frac{1}{k}}\right\}\frac{1}{\cos(\theta) }\sum_{i=1}^{M(\delta)}\Hd^{1}(T_i) \nonumber \\
&=\frac{2(1-\lambda)}{\cos(\theta)}\sum_{i=1}^{M(\delta)} \Hd^{1}(T_i).\label{boundt}
\end{align}


\item We now let $\delta\rightarrow 0$ in (\ref{boundt}), since the above process holds for each cover of $A$ we choose according to our procedure. Let $\delta=1/n$ and define the partial sum of the series in (\ref{boundt}) above as $$S_n\equiv \sum_{i=1}^{M(1/n)}\Hd^{1}(T_i).$$

We see that $\{S_n\}$ is monotonically increasing and bounded by $\Hd^{1}(A)<\infty$ except for an error term. This is due to the fact that in the ``overlap'' of the $E_i^\eta$'s, there may be normals to $A$ that intersect more than one $T_i$ with the result that a ``double-counted'' length of at most $2\eta$ will be included in the sum above for each $T_i$. For a given $\delta$, this error term is bounded by $2M(\delta)\eta$. Yet we know that $\eta$ goes to 0 faster than $\delta$ goes to 0 (thus while $M(\delta)<\infty$). Therefore, this error term vanishes as $\delta$ gets small enough, and we have that $\{S_n\}\rightarrow \Hd^{1}(A)$.

Letting $n\rightarrow \infty$ above amounts to letting $\delta\rightarrow0$, which forces $\theta, \lambda \rightarrow0$ by our construction above. This gives

\begin{align*}\label{result}
\lim_{\delta\rightarrow0}\left\{\frac{2(1-\lambda)}{\cos(\theta)}\sum_{i=1}^{M(\delta)}\Hd^{1}(T_i)\right\}
=2\Hd^{1}(A).
\end{align*}

\noindent Thus we have shown that 
$$\lim_{k\rightarrow \infty}\left\{ \frac{1}{k}\int_{B(0,2R)}\left(\frac{|Df|}{|f|}\right)^{\frac{k-1}{k}}dx\right\}\leq2\mathcal{H}^{1}(A).$$

\item Using an analogous reasoning with the bound in (\ref{etaxlower}), we obtain the desired result.

\end{enumerate}
\end{proof}

\begin{rem}
The next reasonable step is to consider whether Theorem \ref{C1} holds for rectifiable boundaries. While it holds for polygonal boundaries, a relatively simple example shows that simply being rectifiable is not enough (see the Appendix). What about for the boundaries of closed sets of finite perimeter? This question is more delicate and seems best to start with closed connected sets of finite perimeter, whose boundaries are also connected, such that the topological boundary equals the measure theoretic boundary. Let $E$ be such a set. By Theorem 2.7 in \cite{evans-2015-measure}, we see that if $B(x,r)$, $r>0$, is a ball centered on $x\in \partial E$, then $\partial E \cap B(x,r)$ must be quasi-graph like. That is, $\partial E$ comes into the ball once, passes through $x$, and leaves the ball once. A sensible next step is to prove Theorem \ref{C1} holds for such a set $E$ whose boundary is locally Lipschitz, i.e. locally $\partial E$ is the graph of a Lipschitz function.  We leave this for future work. 

\end{rem}
\section{Appendix}
\begin{exm}
Here we construct an example that shows that Theorem \ref{C1} does not hold in general for sets with rectifiable boundaries. Let $S$ be the open unit square in $\R^2$ and let $B(r)$ be the disc of radius $r>0$.  First, we construct a sequence $\{r_m\}$ with the property such that $$\frac{|B(r_{m+1})| +|B(r_{m+2})|+|B(r_{m+3})|+\dots}{|B(r_m)|}\rightarrow 0\, \,  \, \text{as } m\rightarrow \infty.$$ 

Start with $\{s_k\}=\frac{1}{2},\frac{1}{4},\frac{1}{8},\dots$ and let  $\{\epsilon_m\}=\frac{1}{2},\frac{1}{4},\frac{1}{8},\dots$. From $\{s_k\}$, we want to construct a sequence $\{s_{k(m)}\}\equiv\{r_m\}$ such that for each $m$ $$\frac{|B(r_{m+1})| +|B(r_{m+2})|+|B(r_{m+3})|+\dots}{|B(r_m)|} < \epsilon_m,$$ 
 and then let $\{\epsilon_m\}\rightarrow 0.$ Let $ \{r_1\}=\{s_1\}=\frac{1}{2}$ be the first point in $\{r_m\}$, but let $ \{r_2\}=\{s_3\}=\frac{1}{8}$ be the second point. This gives
 $$\frac{\frac{1}{8^2}+\frac{1}{16^2}+\frac{1}{32^2}+\dots }{\frac{1}{2^2}} < \frac{1}{2},$$
which follows from the fact that without squaring, we have equality, and squaring only reduces the LHS. 
 
 Now let $ \{r_3\}=\{s_6\}=\frac{1}{64}$ be the third point. Note that the previous inequality is improved by removing $\frac{1}{16}$ and $\frac{1}{32}$. This time, for $m=2$, we have  $$\frac{\frac{1}{64^2}+\frac{1}{128^2}+\frac{1}{256^2}+\dots }{\frac{1}{8^2}} < \frac{1}{4}.$$ We continue in this way, skipping first one point of $ \{s_k\}$, then two, then three, and so on, so that  $\{r_m\}=\frac{1}{2}+\frac{1}{8}+\frac{1}{64}+\frac{1}{1024}+\dots.$ In general, for the $m$th point, we choose $r_m=s_{(\frac{m^2+m}{2})}$. Observe that picking any subsequence $\{r_{m_j}\}\subset\{r_m\}$ also has the above property.

Now,  we enumerate the rational points in $S$ as $\{q_j\}_{j=1}^\infty$, pick the first rational point $q_1$, and draw the closed disc $B(q_1,r_{m_1})$ with $r_{m_1}$  chosen from $\{r_m\}$ such that $B(q_1,r_{m_1})\cap \partial S =\emptyset$ . Let  $C_1\equiv\partial B(q_1,r_{m_1})$. Pick the next  rational $q_2$. If $q_2\notin C_1$, we pick a smaller radius $r_{m_2}<r_{m_1}$ from our sequence $\{r_m\}$ so that
$$B(q_2,r_{m_2})\cap \{C_1 \cup \partial S\}=\emptyset,$$ and 
$$3r_{m_2}<d(q_2, C_1 \cup \partial S),$$ where $d(q_2, C_1 \cup \partial S)$ is the distance function from the point $q_2$ to the set $C_1 \cup \partial S$. If $q_2\in C_1$, then we perturb $q_2$ slightly off of $C_1$ so that our new point $\oldhat{q}_2\notin C_1\cup \partial S.$ We  pick $r_{m_2}<r_{m_1}$ so that
$$B(\oldhat{q}_2,r_{m_2})\cap \{C_1 \cup \partial S\}=\emptyset,$$ and 
$$3r_{m_2}<d(\oldhat{q}_2, C_1 \cup \partial S).$$

Define $\mathcal{B}_1\equiv B(q_1,r_{m_1})$ and let the following also hold with $\oldhat{q}_2$ in place of $q_2$. If the disc $B(q_2,r_{m_2})\subset \mathcal{B}_1$, then we let $B^o(q_2,r_{m_2})$ be the open disc, and we replace $\mathcal{B}_1$ with $$\mathcal{B}_2\equiv \mathcal{B}_1\setminus B^o(q_2,r_{m_2}).$$ Otherwise, $$\mathcal{B}_2\equiv \mathcal{B}_1\cup B(q_2,r_{m_2}).$$

Letting  $C_j\equiv\partial B(q_j,r_{m_j})$, we can repeat this process for any rational point in our list so that if $q_j\notin \cup_{i=1}^{j-1}C_i$, we pick  $r_{m_j}<r_{m_{j-1}}$ from  $\{r_m\}$ so that $$B(q_j,r_{m_j})\cap \{ \cup_{i=1}^{j-1}C_i \cup \partial S\}=\emptyset,$$ and $$3r_{m_j}<d(q_j, \cup_{i=1}^{j-1}C_i \cup \partial S).$$

If $q_j\in \cup_{i=1}^{j-1}C_i$, then we perturb $q_j$ in the following way. Draw the ball $B(q_j,r_{m_{j-1}})$, which by construction will only intersect the circle containing $q_j$. Since we have a finite number of circles in $\mathcal{B}_{j-1}$ for all $j$, we know that $\cup_{i=1}^{j-1} C_i\cup \partial S$ is a closed set and thus its   complement is open. Then we can pick a point $\oldhat{q}_j\in B(q_j,r_{m_{j-1}})$ in the complement and a radius $r_{m_j}$ from our sequence such that  $r_{m_j}<{r_{m_{j-1}}}$   and $$3r_{m_j}<d(\oldhat{q}_j,\cup_{i=1}^{j-1} C_i \cup \partial S).$$ 

Again, letting the following also hold with $\oldhat{q}_j$ in place of $q_j$, if $$B(q_j,r_{m_j})\subset \mathcal{B}_{j-1},$$ then we replace $\mathcal{B}_{j-1}$ with $$\mathcal{B}_j\equiv \mathcal{B}_{j-1}\setminus B^o(q_j,r_{m_j}).$$ 

Otherwise, $$\mathcal{B}_j\equiv \mathcal{B}_{{j-1}}\cup B(q_j,r_{m_j}).$$ 

Continuing in this way, we construct a set $\mathcal{B}$ such that $S\subseteq \clos({\mathcal{B}})$ and $\partial^* \mathcal{B}\equiv \cup_{i=1}^\infty C_i $, where $\partial^* \mathcal{B}$ is the reduced boundary of $ \mathcal{B}$. 

Lastly, we pick a point $x$ on any circle $C_i$, pick $r_{m_j}$ smaller than the radius of $C_i$, and draw $B(x,r_{m_j})$. We know by construction that $B(x,r_{m_j})\cap C_k=\emptyset$ as long  as $i\ge j$, except when $k=i$. Therefore $\mathcal{B}\cap B(x,r_{m_j}) \cap B(q_{p}, r_{m_{p}})$ is either  $\mathcal{B}\cap B(x,r_{m_j})$ or it is the empty set so long as $p\ge j$ and $p\neq i$. We see that the density ratios computed to get the approximate normal are perturbed only by discs of radius $r_{m_{j+n}}$, where $n\ge 1$. Using our subsequence $\{r_{m_j}\}$ and the property of our ratio above, we have that $$\frac{|B(r_{m_{j+1}})| +|B(r_{m_{j+2}})|+|B(r_{m_{j+3}})|+\dots}{|B(x,r_{m_j})|} <\epsilon_{m_j}.$$ Letting ${m_j} \rightarrow \infty$, we see that $x$ is on the reduced boundary  of the set of circles $\cup_{i=1}^\infty C_i $ since there exists an approximate normal there. By construction, $\mathcal{B}$ is a set of finite perimeter and is thus rectifiable by Theorem 5.15 in \cite{evans-2015-measure}. Thus every point in $S$ is a limit point of $\cup_{i=1}^\infty C_i $, which implies the distance function from any point in $S$ to $\cup_{i=1}^\infty C_i $ is 0 and Theorem \ref{C1} does not hold.

\end{exm}


      



      



\bibliographystyle{plain}
\bibliography{boundarymeasure.bib}

\begin{thebibliography}{1}

\bibitem{evans-2015-measure}
Lawrence~C Evans and Ronald~F Gariepy.
\newblock {\em Measure theory and fine properties of functions}, volume~5.
\newblock CRC press, revised edition, 2015.

\bibitem{federer1959curvature}
Herbert Federer.
\newblock Curvature measures.
\newblock {\em Transactions of the American Mathematical Society},
  93(3):418--491, 1959.

\bibitem{janich2013vector}
Klaus J{\"a}nich.
\newblock {\em Vector analysis}.
\newblock Springer Science \& Business Media, 2013.

\bibitem{la2001characterization}
Davide La~Torre, Matteo Rocca, et~al.
\newblock A characterization of {$C^{k,1}$} functions.
\newblock {\em Real Analysis Exchange}, 27(2):515--534, 2001.

\bibitem{luc1995taylor}
Dinh~The Luc.
\newblock Taylor's formula for {$C^{k,1}$} functions.
\newblock {\em SIAM Journal on Optimization}, 5(3):659--669, 1995.

\bibitem{morgan-2008-geometric}
Frank Morgan.
\newblock {\em Geometric measure theory: a beginner's guide}.
\newblock Academic Press, fourth edition, 2008.

\bibitem{vixie-geo-deriv-tubes}
Kevin~R. Vixie.
\newblock Geometry from deriviatives: from edges to tubes.
\newblock Presented May 17, 2013.

\bibitem{vixie-geometric-derivative}
Kevin~R. Vixie.
\newblock An invitation to geometric measure theory: part 1.
\newblock {https://notesfromkevinrvixie.org/category/mathematics/}.

\end{thebibliography}


\end{document}